\documentclass{article}

\usepackage{inputenc}
\usepackage{amsmath, dsfont}
\usepackage{times}
\usepackage[margin = 1in]{geometry}
\usepackage{color}
\usepackage{comment}
\usepackage{enumitem}
\usepackage{verbatim}
\usepackage[noend]{algorithmic}

\usepackage[colorlinks,urlcolor=blue,citecolor=blue,linkcolor=blue]{hyperref}

\usepackage{natbib}
\usepackage{amssymb, amsmath, amsthm}
\usepackage{amsfonts,mathtools,mathrsfs,mathtools,color}
\usepackage{graphicx}
\usepackage{caption}
\usepackage{subcaption}
\usepackage{bbm} 
\usepackage{algorithm} 
\usepackage{algorithmic} 
\usepackage{comment}
\usepackage{xcolor}
\usepackage{url}
\usepackage{mdframed}
\usepackage{soul}
\usepackage{blindtext}
\usepackage{titlesec}
\usepackage{multirow}
\usepackage{wrapfig}

\newtheorem{theorem}{Theorem}

\newtheorem*{theorem*}{Theorem}
\newtheorem{lemma}{Lemma}

\newcommand\hf{\mathsf{HF}}
\newcommand\Hd{{\em Hamiltonian Descent}\xspace}

\newcommand{\p}[2]{\xi_{#1,#2}}
\newcommand{\q}[2]{\zeta_{#1,#2}}
\usepackage{color}

\def\reals{\mathbb{R}}

\def\HF{{\em Hamiltonian Flow}\xspace}

\def\E{\mathbb{E}}

\title{Frictionless Hamiltonian Descent and Coordinate Hamiltonian Descent for Strongly Convex Quadratic Problems}
\author{Jun-Kun Wang\footnote{
            University of California San Diego, 
            \texttt{jkw005@ucsd.edu}}
            }
\date{}

\begin{document}
\maketitle

\begin{abstract}
We propose an optimization algorithm called Frictionless Hamiltonian Descent, which is a direct counterpart of classical Hamiltonian Monte Carlo in sampling. 
We analyze Frictionless Hamiltonian Descent for strongly convex quadratic functions 
and show that the method has a non-trivial accelerated rate 
comparable to that of Heavy Ball flow.
We also propose Frictionless Coordinate Hamiltonian Descent and its parallelizable variant, which together encompass the classical Gauss-Seidel method, Successive Over-relaxation, Jacobi method, and more, for solving a linear system of equations. The result not only offers a new perspective on these existing algorithms but also leads to a broader class of update schemes that guarantee convergence.
\end{abstract}

\section{Introduction}

The laws of classical mechanics, which describe the dynamics of moving bodies in physical space,
have been a subject of great practical interest and inspired centuries of deep mathematical study. 
In particular, the perspective that describes the dynamics of a particle as a Hamiltonian Flow in phase space has found useful applications in designing algorithms such as Hamiltonian Monte Carlo \citep{Duane87,neal2011mcmc,wang2022accelerating}, a popular sampling algorithm in statistics and Bayesian machine learning \citep{HG14,salvatier2016probabilistic,carpenter2017stan}. In this work, we ask whether there exists a natural counterpart of Hamiltonian Monte Carlo in optimization. More specifically, we consider 
designing algorithms based on the Hamiltonian Flow for efficiently solving an optimization problem, i.e., solving
$\min_{x \in \reals^d} f(x).$

In classical mechanics, a particle is oftentimes described by its position $x \in \reals^d$ and its velocity $v \in \reals^d$ with 
the Hamiltonian 
consisting of its potential energy and its kinetic energy. 
In this paper, we define the Hamiltonian $H(x,v)$ as 
$H(x,v):= f(x) + \frac{1}{2} \| v \|^2_2,$
where the potential energy $f(\cdot)$ is the underlying objective function that we aim to minimize and $\frac{1}{2} \| v \|^2_2$ is the kinetic energy.
Hamiltonian dynamics is one of the fundamental tools in classical mechanics for describing many physical phenomena. The dynamics of the \emph{Hamiltonian Flow} is governed by a system of differential equations:
\begin{equation} \label{flow}
\frac{dx}{dt} = \frac{\partial H}{\partial v} = v  \quad \text{ and } \quad \frac{dv}{dt} = - \frac{\partial H}{\partial x } = - \nabla f(x).
\end{equation}
Given the initial position $x_0$ and velocity $v_0$ at $t=0$, the solution to \eqref{flow}  describes the position $x_{\eta}$ and the velocity $v_{\eta}$ of a particle 
at the integration time $t=\eta$. Throughout this paper, we express the execution of the Hamiltonian Flow for a duration $\eta$ from the initial condition $(x_0,v_0)$ as
$(x_{\eta},v_{\eta}) = \hf_{\eta}(x_0, v_0)$. 

\begin{algorithm}[t]
\begin{algorithmic}[1]
\caption{\textsc{Frictionless Hamiltonian Descent (Frictionless-HD)}
} \label{alg:opt}{}
\STATE \textbf{Input:} an initial point $x_{1} \in \reals^{d}$, number of iterations $K$, a scheme of integration time $\{\eta_k\}$.
\FOR{$k=1$ to $K$}
\STATE $(x_{k+1},v_{k+1}) = \hf_{\eta_k}(x_{k}, 0)$.
\ENDFOR
\STATE \textbf{Output:} $x_{K+1}$.
\end{algorithmic}
\end{algorithm}

Equipped with the Hamiltonian Flow \eqref{flow}, we propose \textit{Frictionless-}\Hd (Algorithm~\ref{alg:opt}), where
at each iteration $k$, a Hamiltonian Flow for a duration $\eta_k$ is conducted from the current update $x_k$ with zero initial velocity. 
It turns out that Frictionless Hamiltonian Descent is a ``descent'' method. 
The property of being a descent method can be derived from the fact that the Hamiltonian is conserved along the Hamiltonian Flow, a classical result in Hamiltonian mechanics,
e.g., \citet{Arnold1989,greiner2003classical,neal2011mcmc}. 
\begin{lemma} \label{lem:conserve}
The time derivative of the Hamiltonian satisfies 
$\frac{dH}{dt} = 0$ along the Hamiltonian Flow \eqref{flow}.
\end{lemma}
\begin{proof}
By the chain rule, 
$\frac{dH}{dt} = \left \langle \frac{\partial H}{\partial x} , \frac{ \partial x}{ \partial t}  \right \rangle 
+
 \left \langle \frac{\partial H}{\partial v} , \frac{ \partial v}{ \partial t}  \right \rangle
  = \left \langle \nabla f(x), v \rangle + \langle v, -\nabla f(x) \right \rangle = 0 \notag.$
\end{proof}
What Lemma~\ref{lem:conserve} implies is that we have a nice property: 
\begin{equation*} \label{eng}
\textbf{(Energy Conservation): } 
H(x_t,v_t) = H(x_0,v_0), \forall t \geq 0,
\end{equation*}
where $(x_{t},v_{t}) = \hf_{t}(x_0, v_0)$. 
In other words, the value of the Hamiltonian at any time $t$ is the same as its initial value along the Hamiltonian Flow. 
Based on the conservation, 
Frictionless-HD (Algorithm~\ref{alg:opt}) has 
\begin{equation*} \label{descent2}
f(x_{k+1}) + \frac{1}{2} \|v_{k+1}\|^2_2 = H(x_{k+1},v_{k+1}) = H(x_k, 0) = f(x_k)
\end{equation*}
where we used Lemma~\ref{lem:conserve} for the second equality.
Since the kinetic energy is non-negative, the above equality implies that
Frictionless Hamiltonian Descent decreases the function value as long as the kinetic energy is non-zero.

Frictionless Hamiltonian Descent can be considered as the direct counterpart of Hamiltonian Monte Carlo (HMC).  
The only difference between their updates is that in HMC, the initial velocity for running a Hamiltonian Flow is randomly sampled from the normal distribution, while in Frictionless Hamiltonian Descent, the initial velocity is set to be zero before running the flow at each iteration. 
The reset of the velocity (and hence resetting the kinetic energy to zero) is a critical operation for Frictionless Hamiltonian Descent to be an optimization method. If the value of the potential energy $f(\cdot)$ is at a minimum but the kinetic energy is not zero due to energy conservation, then the update will keep moving and may not be able to converge.
It is worth mentioning the restart schemes such as \citet{alamo2022restart, necoara2019linear} in the literature, which are designed to transform sublinear-rate algorithms into accelerated linear-rate methods. These schemes typically leverage specific properties of the objective function, such as the quadratic growth condition, by restarting the underlying sublinear-rate optimization algorithm. The core idea is to exploit the fast progress often observed during the initial iterations of a sublinear-rate optimization method (see, e.g., Section 6.1 in \citet{DST21} for an illustration). Therefore, the motivations and mechanisms behind those schemes differ from Frictionless Hamiltonian Descent.

We note that the use of toolkits from Hamiltonian Mechanics in designing and analyzing various optimization algorithms has been quite popular in the optimization literature, e.g., \citet{suBoydCandes2014_differential,krichene2015accelerated,wibisono2016variational,HL17,SRBA17,betancourt2018symplectic,AttouchChbaniPeypouquetRedont2018_fast,BotCsetnek2018_convergence,
maddison2018hamiltonian,FrancaRobinsonVidal2018_admm,muehlebach2019dynamical,zhang2018direct,DO19,o2019hamiltonian,francca2020conformal,
diakonikolas2021generalized,FrancaJordanVidal2021_dissipative,ZhangOrvietoDaneshmand2021_rethinking,AGV21,DST21,
wilson2016lyapunov,
muehlebach2021optimization,even2021continuized,shi2022understanding,
SuhRohRyu2022_continuoustime,de2022born,attouch2022first,
dixit2023accelerated,
WW23a,akande2024momentum}. 
However, what sets our work apart from the aforementioned works is that we consider a pure Hamiltonian dynamical system for optimization between each reset, which, to the best of our knowledge, is rare in the field of optimization \citep{teel2019first,diakonikolas2021generalized}. Specifically, most existing works focus on friction-based optimization methods when leveraging techniques or principles from mechanics in designing and analyzing  algorithms. 
For example, \citet{maddison2018hamiltonian} consider 
a \emph{friction-based} Hamiltonian optimization method:
\begin{equation} \label{damping}
\frac{dx}{dt} = \nabla \phi(v)  \quad \text{ and } \quad \frac{dv}{dt} = - \nabla f(x) - \gamma v,
\end{equation}
where $\gamma > 0$ is the damping parameter that induces friction throughout the optimization process. The dynamic \eqref{damping} reduces to the continuous-time Heavy Ball dynamic \citep{P64} when $\nabla \phi(v)=v$. 

Indeed, a common belief in optimization is that to decrease a function value quickly, dissipating energy via friction along the way might be necessary. The inclusion of the friction term into the optimization dynamics has been the cornerstone of designing accelerated methods in optimization \citep{suBoydCandes2014_differential,krichene2015accelerated,wibisono2016variational,HL17,SRBA17,maddison2018hamiltonian,
muehlebach2019dynamical,WLA21,shi2022understanding,wang2024no}.  
However, surprisingly, on strongly convex quadratic problems, which is one of the most canonical optimization problems that has consistently served as a great entry point for understanding optimization behavior of an algorithm \citep{NB15,SP20,AGZ21,paquette2021sgd,kelner2022big,goujaud2022super}, we find that the \emph{Frictionless} Hamiltonian Descent (Algorithm~\ref{alg:opt}) achieves an accelerated rate.  
Even more, 
when compared with Polyak's Heavy Ball flow
for strongly convex functions \citep{P64,wilson2016lyapunov,sanz2021connections,
aujol2023convergence,kassing2024polyak,okamura2026heavy} with resets, i.e., the flow governed by
\begin{equation}   \label{HBODE}
\ddot{x} + 2 \sqrt{m} \dot{x} + \nabla f(x) = 0,  
\end{equation}  
where $m$ is the strong convexity constant of the underlying function $f(\cdot)$, we show that the \emph{Frictionless} Hamiltonian Descent achieves the same accelerated rate.
This result is unexpected, since it is generally believed that the incorporation of the friction term \(2 \sqrt{m} \dot{x}\) helps  quickly decrease the objective value and converge faster than standard gradient descent. Yet, our finding is that running a pure Hamiltonian Flow \eqref{flow} with resets (Algorithm~\ref{alg:opt}) can be competitive with Heavy Ball flow, which incorporates the friction. 

To summarize, our contributions include: 
\begin{itemize}
\item We propose Frictionless Hamiltonian Descent (Frictionless-HD) and establish its descent property by showing that the successive reset iterates have non-increasing objective values. We analyze Frictionless-HD for strongly convex quadratic functions 
and demonstrate that Frictionless-HD enjoys a non-trivial accelerated rate as that of Heavy Ball ODE for strongly convex quadratic problems.
\item We further propose Frictionless Coordinate Hamiltonian Descent (Frictionless-CHD). We find that for solving a system of linear equations, classical algorithms such as Gauss-Seidel and Successive Over-relaxation are instances of Frictionless-CHD with different schemes of integration time. This new perspective not only provides a simple proof of the convergence of Successive Over-relaxation but also yields a broad class of updates that are guaranteed to converge. 
\item Furthermore, we introduce Parallel Frictionless Coordinate Hamiltonian Descent, which, under different choices of the integration time, becomes the Jacobi method and the weighted Jacobi iteration for solving a linear system of equations. The perspective of Hamiltonian dynamics also helps identify new parallelizable variants and conditions that guarantee convergence.
\end{itemize}

\noindent
\textbf{More related works.} 
To the best of our knowledge, only a couple of previous works have considered optimization based on the principles of Frictionless Hamiltonian Flow with the corresponding dynamics \eqref{flow}.
\citet{teel2019first} consider optimization based on the Hamiltonian dynamics and propose resetting the velocity of Hamiltonian Flow to zero
whenever the update $x$ is about to exit a set defined as $\{ x \in \reals^d, v \in \reals^d: \langle \nabla f(x), v \rangle \leq 0 \, \& \, \| v \|^2 \geq \| \nabla f(x) \|^2 / L \}$ or when a timer has timed out, where $L$ is the smoothness constant of $f(\cdot)$.
Although also built upon the Hamiltonian Flow, their algorithm differs from ours.
A uniform global stability convergence result is given in \citep{teel2019first} for minimizing smooth strongly convex functions using their proposed algorithm. 
\citet{diakonikolas2021generalized} nicely show that along the Hamiltonian Flow,
the norm of time-averaged gradient, i.e.,
$\| \frac{1}{t} \int_0^t \nabla f(x_{\tau}) d\tau \|$, decreases at an $O(1/t)$ rate.
Moreover, they consider a time-varying Hamiltonian $H(x,v,\tau) = h(\tau) f(x/\tau) + \psi^*(v)$, where $h(\tau)$ is a positive function of a scaled time $\tau$ and $\psi(\cdot)$ is a strongly convex differentiable function. They show that a broad class of momentum methods
in Euclidean and non-Euclidean spaces can be produced from the equations of motion, which include
classical Nesterov's accelerated method \citep{nesterov1983method,nesterov2018lectures,wibisono2016variational,wang2018acceleration} and Polyak's Heavy Ball \citep{P64} as special cases.
We also note that \citet{de2022born} consider dynamics motivated by mechanics and propose an algorithm which requires randomly specifying a non-zero velocity at each step to restore a notion of energy in their proposed algorithm. 

After the preprint of this work appeared on arXiv, \citet{fu2025hamiltonian} proposed a randomized integration-time scheme for Frictionless Hamiltonian Descent, in which the integration time at each iteration is sampled from an exponential distribution. They also introduced a discretization scheme for simulating the Hamiltonian flow, yielding a discrete-time optimization algorithm with iteration-complexity guarantees for general smooth strongly convex functions and smooth nonconvex functions. Furthermore, \citet{pmlr-v336-wang26c} established a notable result that a deterministic integration time suffices to achieve acceleration. Another related work, \citet{zhu2026frictionless}, considers alternative choices of the kinetic energy function and a parallel-in-time technique for parallelizing the algorithm’s inherently sequential updates.

\section{Theoretical analysis of Frictionless Hamiltonian Descent for strongly convex quadratic problems} \label{sec:Hami}

We continue by analyzing Frictionless-HD (Algorithm~\ref{alg:opt}) for minimizing strongly convex quadratic functions, i.e., solving
\begin{equation} \label{quadratic}
\min_{x \in \reals^d} f(x), \quad \text{where } f(x):=\frac{1}{2} x^\top A x - b^\top x,
\end{equation}
where $A \succ 0$ is a positive definite matrix. 
To get the ball rolling, in the following, we denote
 $\cos\left( Q \right)$ the trigonometric function of a square matrix $Q \in \reals^{d\times d}$, 
which is defined as 
$\cos\left( Q \right)  := \sum_{i\geq0} \frac{(-1)^i  Q^{2i}}{ (2i)! }$.  
Furthermore, we denote $x_*:=A^{-1}b$ the optimal point of $\eqref{quadratic}$. 
\begin{lemma} \label{app:1}
Let $f(x) :=\frac{1}{2} x^\top A x - b^\top x $.
Apply the \HF \eqref{flow} to $f(\cdot)$ with an initial position $x_0 \in \reals^d$ and an initial zero velocity $0 \in \reals^d$.
Then, $(x_{\eta},v_{\eta}) = \hf_{\eta}(x_0,0)$ has a closed-form expression, i.e.,
$  x_{\eta} =  A^{-1} b  + \cos( \eta \sqrt{A} ) \left( x_0 - A^{-1} b  \right)$
and 
$  v_{\eta} = -\sqrt{A} \sin (\eta \sqrt{A} ) \left( x_0 - A^{-1} b  \right), $
where $\sqrt{A}$ represents the matrix square root of $A$, i.e., $\sqrt{A} \sqrt{A} = A$.
\end{lemma}

\begin{proof}
We have
$\dot{x} = v$ and
$\dot{v} = - \left( A x - b \right).$
In the following, we denote 
$y_t := x_t - A^{-1} b.$
Then, we have
$\dot{y} = \dot{x} = v$ and $\dot{v} = - A y$,
which can be re-written as
 $\begin{bmatrix} \dot y  \\  \dot v  \end{bmatrix} = \begin{bmatrix} 0 & I_d \\ - A & 0 \end{bmatrix}
 \begin{bmatrix}  y  \\  v  \end{bmatrix}.$
Then, 
at integration time $t=\eta$, the solution to the differential equation is 
$\begin{bmatrix}  y_{\eta}  \\  v_{\eta}  \end{bmatrix}
= \exp \left( \begin{bmatrix} 0 & \eta I_d \\ - \eta A & 0 \end{bmatrix}  \right) \begin{bmatrix} y_0 \\ v_0 \end{bmatrix}.$
By noticing that
\begin{align}
\exp \left( \begin{bmatrix} 0 & \eta I_d \\ - \eta A & 0 \end{bmatrix}  \right) 
& =
\begin{bmatrix}
\cos ( \eta \sqrt{A} ) & \frac{1}{\sqrt{A}} \sin (\eta \sqrt{A}) \notag
\\ 
-\sqrt{A} \sin (\eta \sqrt{A} ) & \cos ( \eta \sqrt{A})
\end{bmatrix},
\end{align}
we deduce that
$ y_{\eta} =  \cos( \eta \sqrt{A} ) y_0 + \frac{1}{\sqrt{A}} \sin (\eta \sqrt{A}) v_0 $ and
$ v_{\eta} = -\sqrt{A} \sin (\eta \sqrt{A} ) y_0 +  \cos ( \eta \sqrt{A}) v_0$.
Finally, by using the initial condition $v_0=0$ and the expression of $y_{\eta}$,  
we obtain the result. 

\end{proof}

It is worth comparing the dynamic above with the Heavy Ball flow for the same strongly convex quadratic problem. Let $z_t := x_t^{\mathrm{HB}} - x_*$ be the distance at time $t$, where $x_t^{\mathrm{HB}}$ is the position of the Heavy Ball flow governed by \eqref{HBODE} and $x_*$ is the optimal point.
Then, the dynamics of the distance under the Heavy Ball flow for solving \eqref{quadratic} are given by:
\begin{equation} \label{HB}
z_t
=
e^{-\sqrt m t}
\Big(
\cos(\sqrt{A-mI}\,t)\,  z_0
+
(\sqrt{A-mI})^{-1}
\sin(\sqrt{A-mI}\,t)
\sqrt{ m}  z_0 
\Big),
\end{equation}
when the initial velocity of the Heavy Ball flow is zero, i.e., $\dot{x}_0=0$,
where we note that 
$\sin(Q) := \sum_{k=0}^{\infty} \frac{(-1)^k Q^{2k+1}}{(2k+1)!}$
for any square matrix $Q$ and that
$Q^{-1} \sin(Q) := \sum_{k=0}^{\infty} \frac{(-1)^k Q^{2k}}{(2k+1)!}.$

\begin{lemma} \label{lem:per_update}
 For the strongly convex quadratic problems \eqref{quadratic}, 
 the update of Frictionless-HD  
satisfies
 $x_{k+1} - x_* = \cos\left( \eta_k \sqrt{ A} \right) \left( x_k - x_* \right)$.
\end{lemma}

\begin{proof}
The proof is 
simply an application of Lemma~\ref{app:1},
where we let $x_{\eta} \gets x_{k+1}$, $x_0 \gets x_k$, and $\eta \gets \eta_k$.
\end{proof}

We remark that, as with the continuous-time Heavy Ball flow in \eqref{HB}, computing the update requires computing $x_* = A^{-1}b$ and the matrix exponential; hence, it is not particularly practical from a computational perspective. However, the continuous-time analysis may serve as a starting point to gain insight into the behavior of its discrete-time counterpart. 
We further note that the dynamic of Frictionless Hamiltonian Descent for strongly convex quadratic functions can be viewed as an optimization counterpart to Hamiltonian Monte Carlo for sampling from a Gaussian distribution \citep{wang2022accelerating}. Later, in Section~\ref{sec:Co-Hami}, we propose Frictionless Coordinate Hamiltonian Descent and its parallel variant, which recover several classical updates and generate new update schemes---all of which feature low per-iteration costs.

While the update is specific to the quadratic problems, 
Lemma~\ref{lem:per_update} sheds some light on why it might not be a good idea for executing the Hamiltonian Flow for an infinite integration time (i.e., not let $\eta \to \infty)$---the dynamic in Lemma~\ref{lem:per_update} suggests that it may periodically return to the initial point if we do not reset the velocity. 
By Lemma~\ref{lem:per_update}, the distance between the update $x_{K+1}$ by Frictionless-HD and $x_*$ can be written as follows:
\begin{lemma} \label{dym:quadratic}
Frictionless-HD (Algorithm~\ref{alg:opt}) for the strongly convex quadratic problems \eqref{quadratic} satisfies
$x_{K+1} - x_* = \left( \Pi_{k=1}^K \cos\left( \eta_k \sqrt{A} \right) \right) (x_1 - x_*).$
\end{lemma}

From the lemma,  we have
$\| x_{K+1} - x_* \|_2 \leq  
\left \|  \Pi_{k=1}^K \cos\left( \eta_k \sqrt{A} \right) \right \|_2
\| x_1 - x_* \|_2$. Denote $m:= \lambda_{\min}(A)$ the smallest eigenvalue and $L:=\lambda_{\max}(A)$ the largest eigenvalue of $A$. 
To further gain insight and an upper bound of the convergence rate, in the following, we define
\begin{equation*}  \label{myroots}
r_{k}^{(K)} := \frac{L+m}{2} - \frac{L-m}{2} \text{cos}\left( \frac{(k-\frac{1}{2}) \pi}{K}  \right), 
\end{equation*}
where $k = 1,2,\dots, K$.
To continue, in the following, we assume $L>m$, since the case of $L=m$ is trivial and can be handled separately.
Specifically, we consider a scaled-and-shifted Chebyshev polynomial, defined as
$\bar{\Phi}_K(\lambda) := \frac{\Phi_K\big(h(\lambda)\big)}{\Phi_K\big(h(0)\big)}$,
where 
$h(\lambda) := \frac{L+m - 2 \lambda}{L-m}$
and $\Phi_K(\cdot)$ is the degree-$K$ Chebyshev polynomial of the first kind.
It is known that $\{r_{k}^{(K)}\}_{k=1}^K$ are the roots of 
the scaled-and-shifted Chebyshev polynomial
$\bar{\Phi}_K(\lambda)$, see e.g., \citet{AGZ21,wang2022accelerating}.
Therefore, $\bar{\Phi}_K(\lambda)$ has an equivalent expression:
\begin{equation*}
\textstyle
\bar{\Phi}_K(\lambda) = \Pi_{k=1}^K \left(1 - \frac{\lambda}{ r_{k}^{(K)} }\right),
\end{equation*}
where $\lambda \in [m, L]$.

Furthermore, let $\mathcal{P}_K$ denote the set of polynomials of degrees at most $K$. 
It is known that $\bar{\Phi}_K$ is the optimal $K$-degree polynomial \citep{vishnoi2013lx,DST21}, i.e.,  
\begin{equation*}
\bar{\Phi}_K = \underset{ P \in P_K, P(0)=1 }{\arg\min}  \underset{\lambda \in [m,L] }{ \max} \left| P(\lambda)  \right|.
\end{equation*}

In the following, we use $\sigma(k;K)$ to denote the $k_{\mathrm{th}}$ element of the array $[1,2,\dots, K]$ \emph{after} any permutation $\sigma$.
\begin{lemma}~\label{lem:key} 
(Lemma 4 in \citet{wang2022accelerating}) 
Denote
$| P_{K}^{\mathrm{Cos}}(\lambda) | :=\left| \Pi_{k=1}^K \mathrm{cos}\left( \frac{\pi}{2} \sqrt{ \frac{\lambda}{ r_{\sigma(k;K)}^{(K)} }  }  \right)   \right| $.
Consider any $\lambda \in [m, L]$.
Then, for any positive integer $K$, we have 
\begin{equation*} \label{sta}
| P_{K}^{\mathrm{Cos}}(\lambda) | \leq 
\left| \bar{ \Phi }_K(\lambda)  \right|.
\end{equation*}
\end{lemma}

\begin{lemma} \label{lem:commute}
The eigenvalues 
of the matrix
$\Pi_{k=1}^K \cos\left(\eta_k \sqrt{A} \right)$
are $\mu_j := \Pi_{k=1}^K \cos\left( \eta_k \sqrt{\lambda_j}  \right)$, for every $j \in [d]$,
where $\lambda_1,\lambda_2,\dots,\lambda_d$ are the eigenvalues of $A$.
\end{lemma}

\begin{proof}
We first show that for all $k \neq k' \in [K]$, the matrices 
$\cos(\eta_k \sqrt{A})$ and $\cos(\eta_{k'} \sqrt{A})$ commute. 
The commutativity, together with the fact that $\cos(\eta_k \sqrt{A})$ 
is symmetric for all $k \in [K]$, implies that the collection 
of matrices $\{\cos(\eta_k \sqrt{A})\}_{k \in [K]}$ is 
simultaneously diagonalizable.

To start, we note that since $A \succ 0$, $A$ admits the eigendecomposition:
$A = U \Lambda U^\top$, where $\Lambda = \mathrm{diag}(\lambda_1, \lambda_2, \dots, \lambda_d)$ is the diagonal matrix of the eigenvalues of $A$ and that $U$ is orthogonal.
Then, $U^\top A U = \Lambda$ and $U^\top \cos(\eta_k \sqrt{A} ) U = \cos( \eta_k \sqrt{\Lambda}), \forall k \in [K]$.
Furthermore, we have
\begin{align} 
\cos(\eta_k \sqrt{A}) \cos(\eta_{k'} \sqrt{A})
&= U \cos(\eta_k \sqrt{\Lambda}) U^\top
   U \cos(\eta_{k'} \sqrt{\Lambda}) U^\top \notag \\
&= U \cos(\eta_k \sqrt{\Lambda})
      \cos(\eta_{k'} \sqrt{\Lambda}) U^\top,  \notag
\end{align}
where we used $U^\top U = I$.
Similarly,
\begin{align} 
\cos(\eta_{k'} \sqrt{A}) \cos(\eta_k \sqrt{A})
=
U \cos(\eta_{k'} \sqrt{\Lambda})
   \cos(\eta_k \sqrt{\Lambda}) U^\top. \notag
\end{align}
Since diagonal matrices commute, we have
\begin{equation} 
\cos(\eta_k \sqrt{\Lambda})
\cos(\eta_{k'} \sqrt{\Lambda})
=
\cos(\eta_{k'} \sqrt{\Lambda})
\cos(\eta_k \sqrt{\Lambda}). \notag
\end{equation}
By combining all of the above, we obtain 
$\cos(\eta_k \sqrt{A}) \cos(\eta_{k'} \sqrt{A})
=
\cos(\eta_{k'} \sqrt{A}) \cos(\eta_k \sqrt{A})$,
which shows the commutativity.
Therefore, we have
\begin{align}
& U^\top \Pi_{k=1}^K \cos(\eta_k \sqrt{A} ) U \notag
\\ & = \underbrace{ U^\top \left(  \cos\left( \eta_K \sqrt{A}    \right)  \right) U }_{ \cos(\eta_K \sqrt{\Lambda} )} 
\underbrace{
U^\top \left(  \cos\left(\eta_{K-1} \sqrt{A}    \right) \right) U }_{ \cos(\eta_{K-1} \sqrt{\Lambda} ) } \cdots
\underbrace{ U^\top \cos\left(\eta_{1} \sqrt{A}    \right) U }_{ \cos(\eta_1 \sqrt{\Lambda} )  } \notag
\\ & = \Pi_{k=1}^K \cos(\eta_k \sqrt{\Lambda} ). \notag
\end{align}
We can now read off the eigenvalues from the diagonal matrix $\Pi_{k=1}^K \cos(\eta_k \sqrt{\Lambda} )$ and get the eigenvalues.

\end{proof}

\begin{lemma} \label{lem:rate} 
(Adapted from Lemma~3 in \citet{wang2022accelerating}; see also Theorem 2.1 in \citet{DST21} and Chapter 16.4 in \citet{vishnoi2013lx}) 
For any positive integer $K$, it holds that
$\max_{\lambda \in [m,L]} \left| \bar{\Phi}_K(\lambda) \right| 
= \frac{2}{ \left( \frac{ \sqrt{\kappa} + 1 }{ \sqrt{\kappa} -1 } \right)^K +
\left( \frac{ \sqrt{\kappa} + 1 }{ \sqrt{\kappa} -1 } \right)^{-K} },$
where $\kappa:= \frac{L}{m}$ is the condition number. 
\end{lemma}

Equipped with the above lemmas, we are ready to show the convergence rate of Frictionless-HD.
\begin{theorem} \label{thm:HD}
Frictionless-HD (Algorithm~\ref{alg:opt}) with the integration time 
$\eta_k = \frac{\pi}{2 } \frac{1}{ \sqrt{ r_{\sigma(k;K)}^{(K)} } }$
for strongly convex quadratic problems \eqref{quadratic} satisfies 
\begin{align}
 \| x_{K+1} - x_* \|_2 & \leq \max_{\lambda \in [m,L]} \left| \bar{\Phi}_K(\lambda) \right|   \| x_{1} - x_* \|_2 \label{16}
 \\ &
 \leq \frac{2}{ \left( \frac{ \sqrt{\kappa} + 1 }{ \sqrt{\kappa} -1 } \right)^K +
\left( \frac{ \sqrt{\kappa} + 1 }{ \sqrt{\kappa} -1 } \right)^{-K} }
    \| x_{1} - x_* \|_2, \label{17}
\end{align}
where $m := \lambda_{\min}(A)$ and $L := \lambda_{\max}(A)$,
for any permutation $\sigma$ of the array $[1,2,\dots,K]$.
\end{theorem}

\begin{proof}

By Lemma~\ref{dym:quadratic}, we have
$  x_{k+1} - x_* =  \cos\left(\eta_k \sqrt{A} \right) (x_k - x_*)$.
Therefore,
\begin{align} 
\| x_{K+1} - x_* \|_2  & =  \| \Pi_{k=1}^K \cos\left(\eta_k \sqrt{A} \right) \left( x_1 - x_* \right) \|_2 \notag \\
& \leq \| \Pi_{k=1}^K \cos\left(\eta_k \sqrt{A} \right) \|_2 \| x_1 - x_*  \|_2  \notag \\
& \leq \max_{j=1,2, \dots, d} \left|  \Pi_{k=1}^K \cos\left( \eta_k \sqrt{\lambda_j}  \right)   \right| \|x_1 - x_* \|_2, \label{a1}
\end{align}
where in the last inequality we used Lemma~\ref{lem:commute}.  
We can upper-bound the cosine product of any $j$ as:
\begin{equation} \label{a2}
\textstyle
\left|\Pi_{k=1}^K \mathrm{cos}\left( \eta_k \sqrt{\lambda_j} \right) \right| 
\overset{(a)}{=}
\left|\Pi_{k=1}^K \mathrm{cos}\left( \frac{\pi}{2} \sqrt{ \frac{\lambda_j}{r_{\sigma(k;K)}^{(K)} } } \right) \right|
\overset{(b)}{\leq} \left| \bar{ \Phi }_K(\lambda_j)  \right| 
\end{equation}
where (a) is because $\eta_k= \frac{\pi}{2} \frac{1}{ \sqrt{r_{\sigma(k;K)}^{(K)}} }$, and (b) is by Lemma~\ref{lem:key}.

Combining \eqref{a1} and \eqref{a2} leads to 
\eqref{16}. Using Lemma~\ref{lem:rate} to further upper-bound 
$\max_{\lambda \in [m,L]} \left| \bar{\Phi}_K(\lambda) \right|$
leads to \eqref{17}.

\end{proof}

Theorem~\ref{thm:HD} can be viewed as an optimization counterpart of a recent  acceleration result of HMC in sampling. Specifically, \citet{wang2022accelerating} propose HMC with the Chebyshev integration time $\eta_k = \frac{\pi}{2 } \frac{1}{ \sqrt{ r_{\sigma(k;K)}^{(K)} } }$
to show a $1 - \Theta\left( \frac{1}{\sqrt{\kappa}}    \right)$ rate in terms of the 2-Wasserstein distance to a target Gaussian distribution in the context of sampling, 
where $\kappa$ is the condition number of the inverse covariance.
It is noted that the scaled-and-shifted Chebyshev polynomial $\bar{\Phi}_K(\lambda)$ appears in the lower-bound of the worst-case convergence rate for some ``first-order'' algorithms \citep{Nemirovski84}, which include Gradient Descent, Heavy Ball \citep{P64}, the Chebyshev method \citep{
golub1961chebyshev}, and more. However, we note that Frictionless-HD is not in the same family as these algorithms, and the worst-case rate should not be compared. 

Now we compare Frictionless-HD and Heavy Ball flow \eqref{HBODE} with resets. To this end, we introduce a notion, which we call the total integration time $T$. It is the \emph{total duration} of a flow over $K$ iterations, i.e., $T = \sum_{k=1}^K \eta_k$, where we recall that $\eta_k$ is the integration time for executing a flow at iteration $k$. 

\begin{theorem} \label{example}
Consider applying Heavy Ball flow \eqref{HBODE} to \eqref{quadratic}.
Denote $m>0$ the least eigenvalue of the underlying matrix $A$.
Then, the total integration time required to converge to an $\epsilon$ error
is $T^{\mathrm{HB}} 
= \tilde{\Omega}\left( \frac{1}{\sqrt{m}}  \right)$,
in the worst case over the initial point.
On the other hand, 
the total integration time $T^{\mathrm{HD}}=
\sum_{k=1}^K \frac{\pi}{2} \frac{1}{
\sqrt{ r_k^{(K)} } }
$ for Frictionless-HD 
is $T^{\mathrm{HD}}=\tilde{\Theta}\left( \frac{1}{\sqrt{m}}  \right)$ asymptotically.
\end{theorem}

\begin{proof}
Based on the update of the Heavy Ball flow for the strongly convex quadratic function \eqref{HB},
we have
\begin{align}
\left \| x_t^{\mathrm{HB}} - x_* \right \|_2
& \leq 
e^{-\sqrt m t}
\left \| \cos\left(\sqrt{A-mI}\,t\right) \right \|_2  \left \| x_0^{\mathrm{HB}} - x_* \right \|_2 \notag
\\ & \quad +  e^{-\sqrt m t} \left \| \left(\sqrt{A-mI}\right)^{-1}  \sin\left(\sqrt{A-mI}\,t \right) \right \|_2 
\sqrt{ m} \left \| x_0^{\mathrm{HB}} - x_*   \right  \|_2. \label{e18}
\end{align}
Denote $\{ \lambda_i\}_{i=1}^d$ the eigenvalues of $A$.
Then, we have that the spectral norm satisfies:
\begin{equation} \label{e19}
\left \| \cos\left(\sqrt{A-mI}\,t\right) \right \|_2 
= \max_{ i \in [d] } \left| \cos\left( \sqrt{\lambda_i -m } t \right)\right| 
 \leq 1, \forall t \geq 0.
\end{equation}
Similarly,
\begin{equation} \label{e20}
\left\| (\sqrt{A-mI})^{-1} \sin(\sqrt{A-mI}\,t) \right\|_{2}
= \max_{i\in[d]}
\frac{\left|\sin(\sqrt{\lambda_i-m}\,t)\right|}{\sqrt{\lambda_i-m}}
\leq t,
\end{equation}
where we used the fact that $\left| \frac{\sin(\theta)}{\theta} \right| \leq 1, \forall \theta \in \reals$.
Combining \eqref{e18}, \eqref{e19}, and \eqref{e20}, we obtain
\begin{equation}
\left \| x_t^{\mathrm{HB}} - x_* \right \|_2
\leq 
e^{-\sqrt{m} t}
\left( 1 + \sqrt{m} t \right)\left \| x_0^{\mathrm{HB}} - x_*   \right  \|_2. \label{e21}
\end{equation}
Using \eqref{e21}, we deduce that, with $K$ resets of the Heavy Ball flow,
\begin{align} 
 \left\| x_{K+1} - x_* \right\|_2
& \leq \left( \prod_{k=1}^K \left( 1 + \sqrt{m}\, \eta_k^{\mathrm{HB}} \right) \right)
\exp\!\left(-\sqrt{m} \sum_{k=1}^K \eta_k^{\mathrm{HB}} \right)
\left\| x_{1}- x_* \right\|_2, \label{e22}
\end{align}
for any scheme of integration time $\{\eta_k^{\mathrm{HB}} \}_{k= 1}^K$.
Below, we provide an example showing that this upper bound is tight.
Therefore, in the worst case over the initialization point, to have the distance $\left\| x_{K+1} - x_* \right\|_2  \leq \epsilon$, the total integration time $T^{\mathrm{HB}} := \sum_{k=1}^K \eta_k^{\mathrm{HB}} 
\geq 
 \frac{1}{\sqrt{m}}
\log \left( \frac{\left\| x_{1}- x_* \right\|_2}{\epsilon} \right)
= \tilde{\Omega}\left( \frac{1}{\sqrt{m}}  \right)$,
where $\tilde{\Omega}(\cdot)$ hides the log factor
$\log \left( \frac{\left\| x_{1}- x_* \right\|_2}{\epsilon} \right)
$.

Now let us analyze Frictionless-HD (Algorithm~\ref{alg:opt}).
We first compute the average integration time over $K$ iterations, which corresponds to $K$ resets. 
The average integration time $\bar{\eta}_K$ over $K$ iterations is: 
$\bar{\eta}_K:= \frac{1}{K} \sum_{k=1}^K \eta_k^{(K)} = \frac{1}{K} \sum_{k=1}^K \frac{\pi}{2} \frac{1}{
\sqrt{ r_k^{(K)} } }$,
where we recall $r_{k}^{(K)} := \frac{L+m}{2} - \frac{L-m}{2} \text{cos}\left( \frac{(k-\frac{1}{2}) \pi}{K}  \right)$, $\forall k \in [K]$, are the roots of the scaled-and-shifted Chebyshev Polynomial. 
An upper bound of the average integration time has been derived in Corollary 1 in \citet{jiang2022dissipation} and Appendix D in \citet{wang2022accelerating}. 
Specifically, they show that when the number of iterations $K$ is sufficiently large, the average integration time is 
$\bar{\eta}_K
= \tilde{\Theta} \left(  \frac{1}{\sqrt{L+m}} \right)$.
This result together with Theorem~\ref{thm:HD} gives us the total integration time to get $\epsilon$-close to $x_*$.
Specifically, from Theorem~\ref{thm:HD}, we know that
the number of iterations $K$ in Frictionless-HD to get an $\epsilon$ distance is $K \leq \frac{ \sqrt{\kappa} + 1 }{2} \log \frac{ 2 \| x_1 - x_* \|  }{\epsilon}:= K_*$.
Therefore, the total integration time asymptotically is 
$T^{\mathrm{HD}} = K_* \times \frac{1}{K_*} \sum_{k=1}^{K_*} \eta_k^{(K_*)} 
= \tilde{ \Theta } \left( \frac{\sqrt{\kappa+1}}{\sqrt{L+m}} \right)
= \tilde{ \Theta } \left( \frac{1}{\sqrt{m}} \right),$
where $ \tilde{ \Theta }\left( \cdot \right)$ hides log factors in $\kappa$ 
and $\|x_1 - x_* \| / \epsilon$.

\end{proof}

We now provide an example to show that the upper bound in \eqref{e21},
and consequently \eqref{e22}, are tight. 
Let 
$
A = \begin{bmatrix} \lambda_1 & 0 \\ 0 & \lambda_2 \end{bmatrix}
$
be a diagonal matrix with $\lambda_1 > \lambda_2 = m > 0$ and $b=[0,0]^\top$.
In this case, $x_* = [0, 0]^\top$.
From the Heavy Ball flow \eqref{HB}, we obtain the following  coordinate-wise dynamics:
$$
\begin{bmatrix} z_t[1] \\ z_t[2] \end{bmatrix}
=
\begin{bmatrix}
e^{-\sqrt{m}t}
\left(
\cos(\sqrt{\lambda_1 - m}\, t)
+
\frac{\sqrt{m}}{\sqrt{\lambda_1 - m}}
\sin(\sqrt{\lambda_1 - m}\, t)
\right)
z_{0}[1]
\\[1.2em]
e^{-\sqrt{m}t}(1+\sqrt{m}t)\, z_{0}[2]
\end{bmatrix}.
$$
From the dynamics above, we observe that the second coordinate of the Heavy Ball flow converges to $0$ at the rate $e^{-\sqrt{m}t}(1+\sqrt{m}t)$, which matches the factor appearing in \eqref{e21} when $z_{0}[2]\neq 0$.

\section{Frictionless Coordinate Hamiltonian Descent} \label{sec:Co-Hami}

In this section, we demonstrate that the optimization principle
of conducting a \emph{Frictionless} Hamiltonian Flow with resets 
can be naturally extended to a (block) coordinate-wise fashion, yielding a few old and new optimization algorithms. In particular, as we are going to show, the choice of when to stop a \emph{coordinate} Hamiltonian Flow can result in different optimization updates, which reveals the potential of the proposed optimization framework.

To begin with, let us denote $[\nabla f(x)]_i \in \reals^d$ as
\begin{equation*}
[\nabla f(x)]_i := \left[  0 , \dots , \nabla f(x)[i] ,  \dots , 0  \right]^\top, 
\end{equation*}
which zeroes out the gradient vector $\nabla f(x) \in \reals^d$ except its $i_{\mathrm{th}}$-element $\nabla f(x)[i] \in \reals$.
Furthermore, we denote $A_{i,j}$ as the element on the $i_{\mathrm{th}}$ row and the $j_{\mathrm{th}}$ column of $A$. 
Then, we consider a flow governed by the following differential equations, 
\begin{equation} \label{CHF}
\frac{d x}{d t} = v \quad \text{ and } \quad \frac{d v}{d t} = - [\nabla f(x)]_i.
\end{equation}
We will use
$(x_{t}, v_{t} ) = \hf_{t}^{(i)}( x_0, 0)$
to denote the execution of the flow defined in \eqref{CHF} for a duration $t$, given 
the initial position $x_0 \in \reals^d$ and the initial zero velocity $0 \in \reals^d$.
This operation outputs the position $x_{t}$ and
the velocity $v_{t}$ of the particle at the integration time $t$.

\begin{lemma} \label{lem:cord_dym}
The solution $(x_t, v_t)= \hf_{t}^{(i)}( x_0, 0)$ is as follows.
For $j \neq i$, $x_t[j] = x_0[j]$ and $v_t[j]= 0$.
For dimension $i$, $x_t[i]$ is the solution to the following one-dimensional differential system: 
$$\frac{dx[i]}{dt} = v[i]  \quad \text{ and } \quad  \frac{dv[i]}{dt} = -\nabla f(x)[i],$$
with the initial position $x_0[i]$ and the initial velocity $v_0[i]$.
\end{lemma}

\begin{proof}
To see this, consider a dimension $j \neq i$,
then, the dynamic is
\begin{equation*}
\frac{dx[j]}{dt} = v[j] \quad \text{ and } \quad \frac{dv[j]}{dt} = 0,
\end{equation*}
which is equivalent to $\frac{d^2 x[j] }{d t^2} = 0$ and hence implies that $\frac{dx[j]}{dt} = v[j]$ is a constant.
By the initial condition $v_0[j]=0$, this means that $v_{t}[j]=0$ for any time $t$ and consequently along the flow $x_{t}[j] = x_0[j]$.

Therefore, the system effectively reduces to a one-dimensional dynamic,
\begin{equation*}
\frac{dx[i]}{dt} = v[i]  \quad \text{ and } \quad  \frac{dv[i]}{dt} = -\nabla f(x)[i]. 
\end{equation*}

\end{proof}

What Lemma~\ref{lem:cord_dym} shows is that when executing the flow \eqref{CHF}, only the $i_{\mathrm{th}}$ element of $x$ is being updated. Hence, we refer to \eqref{CHF} as the \emph{Coordinate Hamiltonian Flow}.
With Lemma~\ref{lem:cord_dym}, we can deduce a crucial property of the Coordinate Hamiltonian Flow, stated as follows.

\begin{lemma} \label{lem:coordinate-wise}
For any $i \in [d]$, the Coordinate Hamiltonian Flow \eqref{CHF} with the zero initial velocity, i.e., 
$(x_{t}, v_{t} )= \hf_{t}^{(i)}( x_0, 0)$, satisfies %
$ f(x_t) + \frac{ v_t^2[i] }{2} = H(x_t,v_t) = H(x_0, v_0) = f(x_0). $
\end{lemma}

\begin{proof}
We have
\begin{align}
\frac{dH(x,v)}{dt} &  = \frac{\partial H(x,v)}{ \partial x }  \frac{\partial x}{ \partial t } +
\frac{\partial H(x,v)}{ \partial v }  \frac{\partial v}{ \partial t } \notag
\\  & \overset{(a)}{=} \nabla f(x)[i] v[i] + \frac{\partial H(x,v)}{ \partial v }  \frac{\partial v}{ \partial t } \notag
\\ & \overset{(b)}{=} \nabla f(x)[i] v[i] - v[i] \nabla f(x)[i] = 0, \notag
\end{align}
where in (a) we used the fact that $v_t[j]=0, \forall t$, $j \neq i$ by Lemma~\ref{lem:cord_dym}, 
and in (b) we used the update $\eqref{CHF}$.
Since the Hamiltonian is conserved, we have the result. 
Hence, $f(x_t) \leq f(x_0)$ along the Coordinate Hamiltonian Flow.

\end{proof}

\begin{algorithm}[t]
\begin{algorithmic}[1]
\small
\caption{\textsc{Frictionless Coordinate Hamiltonian Descent}
} \label{alg:opt3}{}
\STATE Input: an initial point $x_{1} \in \reals^{d}$, number of iterations $K$, and a scheme of integration time $\{\eta_{k,i}\}$.
\FOR{$k=1$ to $K$}
\STATE $x_k^{(0)} \gets x_k \in \reals^d$.
\FOR{$i=1$ to $d$}
\STATE Compute $(x_{k}^{(i)}, v_{k}^{(i)} ) = \hf_{\eta_{k,i}}^{(i)}(x_{k}^{(i-1)}, 0)$. 
\STATE  Set $x_{k+1}[i] = x_k^{(i)}[i]$.
\ENDFOR
\ENDFOR
\STATE Return $x_{K+1}$.
\end{algorithmic}
\end{algorithm}

Lemma~\ref{lem:coordinate-wise} reveals that successive reset iterates have non-increasing objective values. Using this property, we propose Frictionless Coordinate Hamiltonian Descent (Algorithm~\ref{alg:opt3}), which is built upon Coordinate Hamiltonian Flow and is  also a ``descent'' method like HD.
The following lemma, though elementary, will be used multiple times shortly to derive the update of Frictionless Coordinate Hamiltonian Descent and its variants
 and is included here for completeness.

\begin{lemma} \label{lem:1d}
Consider a one-dimensional system with an initial condition $(x_0,v_0=0)$:
$$ \dot{x} = v  \quad \text{ and } \quad \dot{v}  = - \alpha x + \beta.$$
If $\alpha > 0$, then
the solution is 
\begin{align*}
 x_t  = \frac{\beta}{\alpha}  + \cos\left( t \sqrt{\alpha} \right)  \left(x_0 - \frac{\beta}{\alpha} \right)  \quad \text{ and } \quad
 v_t   = -\sqrt{\alpha} \sin\left( t \sqrt{ \alpha } \right) \left(  x_0 - \frac{\beta}{\alpha} \right). 
\end{align*} 
\end{lemma}

\begin{proof}

Let $y:= x - \frac{\beta}{\alpha}$.
Then, we have an equivalent system:
$\dot{y} = v$ and $\dot{v} = - \alpha y$,
which has the solution
\begin{equation*}
y_t = \cos( t \sqrt{\alpha} ) y_0.
\end{equation*}
Substituting $y_t = x_t - \frac{\beta}{\alpha}$, we get the expression of $x_t$ and differentiating $x_t$ w.r.t.~$t$, we get $v_t$.
  
\end{proof}

\begin{lemma} \label{lem:cod}
Define 
$$\p{k}{i}  := \frac{b_i - \sum_{j < i} A_{i,j} x_{k+1}[j] - \sum_{j>i} A_{i,j} x_{k}[j]   }{A_{i,i}}.$$
For solving the quadratic problems \eqref{quadratic} via Frictionless-CHD (Algorithm~\ref{alg:opt3}),
at each inner iteration $i$ of an outer iteration $k$, the update 
$(x_{k}^{(i)}, v_{k}^{(i)} ) = \hf_{\eta_{k,i}}^{(i)}(x_{k}^{(i-1)}, 0)$ is implemented as
\begin{align}
x_{k}^{(i)}[i] & = \p{k}{i} + \cos \left( \eta_{k,i} \sqrt{ A_{i,i} } \right) \left(x_k^{(i-1)}[i] - \p{k}{i} \right) \label{hichd} \\
v_{k}^{(i)}[i] & = - \sqrt{ A_{i,i} } \sin \left( \eta_{k,i} \sqrt{ A_{i,i} } \right) \left(x_k^{(i-1)}[i] - \p{k}{i} \right), \label{hivhd}
\end{align}
and for $j \neq i$, $x_{k}^{(i)}[j]  = x_{k}^{(i-1)}[j]$ and $v_{k}^{(i)}[j]  =0$.  
\end{lemma}

\begin{proof}
We use the notation that $\dot{x}:=\frac{dx}{dt}$ and $\dot{v}:=\frac{dv}{dt}$ in the following. 

By Lemma~\ref{lem:cord_dym}, we have
\begin{equation*}
\dot{x}[i] = v[i]  \quad \text{ and } \quad  \dot{v}[i] = -\nabla f(x)[i] 
= - (Ax-b)[i] 
= - A_{i,i} x[i] + \beta_{i},
\end{equation*}
where $\beta_i:= b[i] - \sum_{j\neq i}^d A_{i,j} x[j] $.
Applying Lemma~\ref{lem:1d} with $\alpha \gets A_{i,i}$, $\beta \gets \beta_i$,  $x \gets x_{k}^{(i-1)}$, $ t \gets \eta_{k,i}$ leads to the result.

\end{proof}

In numerical linear algebra, among the classical iterative methods for solving linear systems of equations are the Gauss-Seidel method and Successive Over-Relaxation \citep{young1954iterative,hackbusch1994iterative,quarteroni2006numerical,iserles2009first}. 
For Gauss-Seidel, at each inner iteration $i$ of an outer iteration $k$, the update is:
\begin{align} \label{gauss}
& \textbf{ Gauss-Seidel: } \quad 
 x_{k+1}[i]  = \p{k}{i}=
\frac{b_i - \sum_{j < i} A_{i,j} x_{k+1}[j] - \sum_{j>i} A_{i,j} x_{k}[j]   }{A_{i,i}}.
\end{align}
It is worth noting that Gauss-Seidel is Coordinate Descent \citep{beck2013convergence,wright2015coordinate} when applied to the quadratic problem \eqref{quadratic}.
On the other hand, Successive Over-Relaxation, which can be provably shown to outperform Gauss-Seidel
when the underlying matrix $A$ satisfies a notion called ``consistently-ordered'' (see e.g., \citet{young1954iterative}, Theorem 5.6.5 in \citet{hackbusch1994iterative}, Theorem 10.10 in \citet{iserles2009first}), 
has the following update:
\begin{align}
& \textbf{ Successive Over-Relaxation: }  \quad  
 x_{k+1}[i]   = c \, \p{k}{i} + (1-c) x_k[i],  \label{hisor}
\end{align}
where $c \in (0,2)$ is a parameter. By comparing \eqref{hichd}, \eqref{gauss}, and \eqref{hisor},
it becomes evident that Gauss-Seidel and Successive Over-relaxation are both instances of Frictionless-CHD
distinguished by their respective values of the integration time $\eta_{k,i}$.
For notational brevity, in the rest of the paper, we denote 
$\cos^{-1}(1-c)$ as any $\theta \in \reals$ such that $\cos(\theta) = (1-c)$, and similarly, $\sin^{-1}(c)$ as any $\theta \in \reals$ such that $\sin(\theta) = c$, where $c \in [0,1]$. 
We formally state the above results as follows. 
\begin{theorem}
The Gauss-Seidel method is 
Frictionless-CHD (Algorithm~\ref{alg:opt3}) with the integration time $\eta_{k,i}=\frac{1}{\sqrt{A_{i,i}}} \sin^{-1}(1)$, 
while Successive Over-relaxation is Frictionless-CHD with the integration time
$\eta_{k,i} = \frac{1}{\sqrt{A_{i,i}}} \cos^{-1}(1-c) $, where $c \in (0,2)$ is a parameter.
\end{theorem}

We now show that Frictionless-CHD asymptotically converges to the global optimal point $x_*$ for any constant integration time
$\eta_{k,i} = \eta_i$, where $\sin(\eta_i \sqrt{A_{i,i}}) \neq 0, \forall i \in [d]$.
\begin{theorem} \label{thm:converges}
For any constant integration time
$\eta_{k,i} = \eta_i$, where $\sin(\eta_i \sqrt{A_{i,i}}) \neq 0, \forall i \in [d]$,
the update $x_k$ by Frictionless-CHD (Algorithm~\ref{alg:opt3}) asymptotically converges to $x_*$, the solution to the strongly convex quadratic problem \eqref{quadratic}.
\end{theorem}

\begin{proof}
By the update of Frictionless-CHD and the conservation of the Hamiltonian (Lemma~\ref{lem:coordinate-wise}), we have
\begin{equation} \label{inner}
 f(x_{k}^{(i-1)}) = f(x_{k}^{(i)}) + \frac{1}{2}  ( v_{k}^{(i)}[i] )^2.
\end{equation}

By a telescoping sum of \eqref{inner} from $i=1$ to $d$, we get
$f(x_{k}) = f(x_{k+1}) + \frac{1}{2} \sum_{i=1}^d  ( v_{k}^{(i)}[i] )^2,$
which means that $\{ f(x_k) \}_{k \geq 1}$ is a non-increasing sequence.
In addition, we have
$ \sum_{k=1}^{\infty} \left( f(x_k) - f(x_{k+1}) \right) =
f(x_1) - \lim_{k \to \infty} f(x_k) < \infty$,
as $f(x_k)$ is lower-bounded by the optimal value $f(x_*)$, which is finite.
However, we also have 
$\sum_{k=1}^{\infty} \left( f(x_k) - f(x_{k+1}) \right) = 
\sum_{k=1}^{\infty}  \sum_{i=1}^d  \frac{1}{2} ( v_{k}^{(i)}[i] )^2$.
Hence, we have
\begin{equation} \label{summable}
\sum_{k=1}^{\infty}  \sum_{i=1}^d  \frac{1}{2} ( v_{k}^{(i)}[i] )^2 < \infty.
\end{equation}
From \eqref{summable}, we conclude that
$\lim_{k \to \infty} v_{k}^{(i)}[i] = 0, \forall i \in [d]$,
which implies that $\lim_{k \to \infty} f(x_k)- f(x_{k+1}) = 0$.

Since the integration times $\eta_{k,i}=\eta_i$ are fixed and satisfy
$\sin(\eta_i\sqrt{A_{i,i}})\neq0$, the expression for the velocity \eqref{hivhd} gives
\begin{equation} \label{hh1}
x_k^{(i-1)}[i]-\xi_{k,i}
=
-\frac{v_k^{(i)}[i]}
{\sqrt{A_{i,i}}\sin(\eta_i\sqrt{A_{i,i}})}
\longrightarrow 0, \, \forall i\in[d],
\end{equation}
where we used $v_k^{(i)}[i]\to0$.
Furthermore, using $x_k^{(i-1)}[i]=x_k[i]$, the coordinate update \eqref{hichd},
and \eqref{hh1},
we obtain
\begin{equation} \label{hh2}
x_{k+1}[i]-x_k[i]
=
\bigl(\cos(\eta_i\sqrt{A_{i,i}})-1\bigr)
\bigl(x_k^{(i-1)}[i]-\xi_{k,i}\bigr)
\longrightarrow 0, \, \forall i\in[d].
\end{equation}
By the definition of $\xi_{k,i}$ (see Lemma~\ref{lem:cod}),
\begin{equation} \label{hh3}
A_{i,i}\bigl(x_k^{(i-1)}[i]-\xi_{k,i}\bigr)
=
A_{i,i}x_k[i]-b[i]
+\sum_{j<i}A_{i,j}x_{k+1}[j]
+\sum_{j>i}A_{i,j}x_k[j].
\end{equation}
Consequently, by \eqref{hh1}, \eqref{hh2}, and \eqref{hh3}, it holds that
\[
(Ax_k-b)[i]
=
A_{i,i}\bigl(x_k^{(i-1)}[i]-\xi_{k,i}\bigr)
+\sum_{j<i}A_{i,j}\bigl(x_k[j]-x_{k+1}[j]\bigr)
\longrightarrow 0.
\]
Since this holds for every $i\in[d]$, we have
$\|Ax_k-b\|_2\to0$.
Finally, because $A$ is invertible and $x_*=A^{-1}b$,
\[
\|x_k-x_*\|_2
=
\|A^{-1}(Ax_k-b)\|_2
\le
\|A^{-1}\|_2\,\|Ax_k-b\|_2
\longrightarrow 0.
\]
This proves that $x_k$ converges to $x_*$.

\end{proof}

An immediate application of Theorem~\ref{thm:converges} provides a simple proof that choosing the parameter $c \in (0,2)$ guarantees the convergence of Successive Over-Relaxation. On the other hand, the original proof in \citet{ostrowski1954linear} demonstrating the convergence of the Successive Over-Relaxation when the parameter satisfies $0 < c < 2$ is quite involved (see also Chapter 3.4 in \citet{varga1962iterative}). 

\begin{theorem} 
If $0< c < 2$, then Successive Over-Relaxation converges to 
$x_*$ of Problem~\eqref{quadratic}.  
\end{theorem}

\begin{proof}
Observe that the condition of Theorem~\ref{thm:converges}, $\eta_{k,i} \neq \frac{1}{ \sqrt{A_{i,i} } } \sin^{-1}(0)$, holds if and only if $\cos\left(\eta_{k,i} \sqrt{A_{i,i} }\right) \neq \pm 1$.
Hence, by choosing $\eta_{k,i}$ such that $\cos\left( \eta_{k,i} \sqrt{A_{i,i} } \right)= 1-c$, where $c \in (0,2)$, we can guarantee the convergence. 

\end{proof}

\noindent
\textbf{Randomized Coordinate Updates:}
We can also consider a randomized version of Frictionless-CHD, as presented in Algorithm~\ref{alg:RHD}. 
In this variant, at each iteration $k$, a coordinate $i \in [d]$ is randomly chosen with probability $p_{i}$ for updating. The update is then performed by executing the Coordinate Hamiltonian Flow on the selected coordinate $i_k \in [d]$ at each iteration $k$, i.e., 
$(x_{k+1}, v_{k+1}) = \hf_{\eta_{k,i_k}}^{(i_k)}(x_{k}, 0).$

\begin{algorithm}[h]
\begin{algorithmic}[1]
\small
\caption{\textsc{Randomized Frictionless Coordinate Hamiltonian Descent}} \label{alg:RHD}{}
\STATE Input: an initial point $x_{1} \in \reals^{d}$, number of iterations $K$, and a scheme of integration time $\{\eta_{k,i}\}$.
\FOR{$k=1$ to $K$}
\STATE Choose $i_k \in [d]$ with probability $p_{k,i}$.
\STATE Compute $(x_{k+1}, v_{k+1}) = \hf_{\eta_{k,i_k}}^{(i_k)}(x_{k}, 0)$. 
\STATE $\text{//} x_{k+1}[i_k] = \phi_{k,i_k} + \cos \left( \eta_{k,i_k} \sqrt{ A_{i_k,i_k} } \right) \left(x_k[i_k] - \phi_{k,i_k} \right)$, \\ \quad where $\phi_{k,i_k} =  \frac{b[i_k] - \sum_{j \neq i_k}^d A_{i_k,j} x_k[j] }{ A_{i_k,i_k}  }  $. 
\STATE $\text{//} x_{k+1}[j] = x_{k}[j], \text{ for } j \neq i_k $.
\STATE $\text{//} v_{k+1}[i_k] = - \sqrt{ A_{i_k,i_k} } \sin \left( \eta_{k,i_k} \sqrt{ A_{i_k,i_k} } \right) \left(x_k[i_k] - \phi_{k,i_k} \right)$. 
\STATE $\text{//} v_{k+1}[j] =0, \text{ for } j \neq i_k$. 
\ENDFOR
\STATE Return $x_{K+1}$.
\end{algorithmic}
\end{algorithm}

\begin{theorem} \label{thm:randomized_d}
Consider Randomized Frictionless Coordinate Hamiltonian Descent (Algorithm~\ref{alg:RHD}). 
Let the integration times $\eta_{k,i}>0$ satisfy
$
\sin^2\left(\eta_{k,i}\sqrt{A_{i,i}}\right)\neq 0, \forall k, \, \forall i\in[d].$
When randomly sampling a coordinate $i \in [d]$ at each iteration $k$ with probability $p_{k,i}  = \frac{A_{i,i}}{ \sin^2(\eta_{k,i}\sqrt{A_{i,i}}) } \frac{1}{G_k}$ to update, 
where $G_k:= \sum_{j=1}^d \frac{ A_{j,j}}{ \sin^2 \left( \eta_{k,j} \sqrt{ A_{j,j} } \right)  }
 $, it satisfies 
$$\textstyle \E[ f(x_{K+1}) -f(x_*)  ] \leq \left( \Pi_{k=1}^K \left( 1 - \frac{\lambda_{\min}(A)}{G_k} \right) \right) \left( f(x_1) - f(x_*) \right). $$
\end{theorem}

\begin{proof}

From Lemma~\ref{lem:coordinate-wise}, the difference between the consecutive iterations satisfies $f(x_k) - f(x_{k+1}) = \frac{1}{2} \left( v_{k+1}[i_k] \right)^2 $, which 
together with the update of $v_{k+1}[i_k] = - \sqrt{ A_{i_k,i_k} } \sin \left( \eta_{k,i_k} \sqrt{ A_{i_k,i_k} } \right) \left(x_k[i_k] - \phi_{k,i_k} \right) $ implies
\begin{align} 
\E_k\left[  f(x_k) - f(x_{k+1})   \right]
& =  \sum_{i=1}^d \frac{1}{2} p_{k,i} A_{i,i} \sin^2 \left( \eta_{k,i} \sqrt{ A_{i,i} } \right) \left(x_k[i] - \phi_{k,i} \right)^2 \notag \\
& = \sum_{i=1}^d \frac{1}{2} p_{k,i} \frac{1}{A_{i,i}} \sin^2 \left( \eta_{k,i} \sqrt{ A_{i,i} } \right) \left( (A x_k)[i] - b[i] \right)^2, \label{77}
\end{align}
where $(A x_k)[i]$ is the $i_{\mathrm{th}}$-element of $A x_k \in \reals^d$ and $\E_k[\cdot]$ is the conditional expectation at $k$. 
Then, from \eqref{77}, we have
\begin{align*}
\E_k\left[  f(x_k) - f(x_{k+1})   \right] = \frac{1}{2 G_k} \| A x_{k} - b \|^2_2,
\end{align*}
where we used the expression of $p_{k,i}$.
Denote $\| x \|_A := \sqrt{ x^\top A x}$.
We have
\begin{equation*}
\frac{1}{2} \| A x_k -b \|^2_2 = \frac{1}{2} \| A x_k - A x_* \|^2_2 \geq  \frac{\lambda_{\min}(A)}{2} \| x_k - x_* \|^2_A 
= \lambda_{\min}(A) ( f(x_k) - f(x_*) )
\end{equation*}
where in the inequality we used $\| x \|_A^2 \leq \| A^{-1} \|_2 \| A x \|^2_2$.
Combining all of the above, we get
\begin{align*}
\E_k[ f(x_{k+1}) -f(x_*)  ] \leq \left(1 -  \frac{ \lambda_{\min}(A) }{G_k} \right) \left( f(x_k) - f(x_*) \right),
\end{align*}
which allows us to conclude the result using the law of total expectation. 

\end{proof}

Since Randomized Coordinate Descent is an instance of Randomized Frictionless Coordinate Hamiltonian Descent, the rate in the above theorem recovers that of Randomized Coordinate Descent \citep{leventhal2010randomized} when $\eta_{k,i_k}= \frac{\pi}{2 \sqrt{A_{i_k,i_k}}}$.

\noindent
\textbf{Block Coordinate Updates:}  Similar to the design of the Frictionless-CHD, we can consider Frictionless Block Coordinate Hamiltonian Descent (Frictionless Block-CHD) based on conducting a block Hamiltonian Flow.
The algorithm is shown in Algorithm~\ref{alg:blockCHD}. 
For this block-coordinate variant, we assume that $A$ is block diagonal with respect to the partition $\{\mathcal B_n\}_{n=1}^N$, i.e., $A_{\mathcal B_n,\mathcal B_m}=0$ whenever $n\neq m$. Since $A\succ 0$, each block $A^{(n)}=A_{\mathcal B_n, \mathcal B_n}$ is positive definite and hence invertible.
When applying Block-CHD to the strongly convex quadratic problems \eqref{quadratic}, the update in each inner iteration $n$ of an outer iteration $k$ is
\begin{align} \label{eq:block}
& x_{k}^{(n)}[ \mathcal{B}_n ] = z^{(n)} + \cos\left( \eta_{k,n} \sqrt{ A^{(n)} }  \right) \left(x_k^{(n-1)} [ \mathcal{B}_n ]   - z^{(n)} \right),
\end{align}
where $z^{(n)}:= ( A^{(n)} )^{-1} b^{(n)}$, and
$| \mathcal{B}_n  |$ is the size of block $\mathcal{B}_n$, 
$A^{(n)} \in \reals^{ | \mathcal{B}_n | \times  | \mathcal{B}_n | }$ is 
the sub-matrix of $A$ indexed by the coordinates in $\mathcal{B}_n$,
$b^{(n)} \in \reals^{|\mathcal{B}_n|}$ is the sub-vector of $b$ indexed by those in $\mathcal{B}_n$, and $x_{k}^{(n)}[ \mathcal{B}_n ] \in \reals^{| \mathcal{B}_n |}$ is the sub-vector of $x_{k}^{(n)}$ indexed by $\mathcal{B}_n$. 

\begin{algorithm}[t]
\begin{algorithmic}[1]
\small
\caption{\textsc{Frictionless Block-Coordinate Hamiltonian Descent (Block-CHD)}
} \label{alg:blockCHD}{}
\STATE Required: splitting the coordinates $[d]$ into $N$ disjoint blocks $\{ \mathcal{B}_n \}_{n=1}^N$. 
\STATE Input: an initial point $x_{1} \in \reals^{d}$, number of iterations $K$, and a scheme of integration time $\{\eta_{k,n}\}$.
\FOR{$k=1$ to $K$}
\STATE $x_k^{(0)} \gets x_k \in \reals^d$.
\FOR{$n=1$ to $N$}
\STATE Compute $(x_{k}^{(n)}, v_{k}^{(n)} ) = \hf_{\eta_{k,n}}^{\mathcal{B}_n}(x_{k}^{(n-1)}, 0)$. 
\STATE  Set $x_{k+1}[i] = x_k^{(n)}[i]$, \quad $\forall i \in \mathcal{B}_n$.
\ENDFOR
\ENDFOR
\STATE Return $x_{K+1}$.
\end{algorithmic}
\end{algorithm}

\begin{lemma}
Apply Frictionless Block-CHD for the strongly convex quadratic problems \eqref{quadratic}. The update $x_{k}^{(n)}$ in each inner iteration $n$ of an outer iteration $k$ is
$$x_{k}^{(n)}[ \mathcal{B}_n ] = ( A^{(n)} )^{-1} b^{(n)} + \cos\left( \eta_{k,n} \sqrt{ A^{(n)} }  \right) \left(x_k^{(n-1)} [ \mathcal{B}_n ]   - (A^{(n)} )^{-1} b^{(n)} \right),$$
\end{lemma}
\begin{proof}
Observe that when executing $(x_{k}^{(n)}, v_{k}^{(n)} ) = \hf_{\eta_{k,n}}^{\mathcal{B}_n}(x_{k}^{(n-1)}, 0)$, the coordinates that are not in $\mathcal{B}_n$ are fixed and are not updated along the block Hamiltonian Flow,
and hence the system of differential equations reduces to the following 
$| \mathcal{B}_n|$-dimensional system of equations,
\begin{equation*}
\frac{dx}{d t} = v \quad \text{ and } \quad \frac{d v}{d t} = - A^{(n)} x + b^{(n)}.
\end{equation*}
Then, by invoking Lemma~\ref{app:1} with $A \gets A^{(n)} $, $b \gets  b^{(n)}$, $\eta \gets \eta_{k,n}$, $x_{\eta} \gets x_{k}^{(n)}[ \mathcal{B}_n ]$,
and $x_0 \gets x_k^{(n-1)} [ \mathcal{B}_n ]$, we obtain the update.

\end{proof}

We note that Coordinate Descent and its variants, as well as their comparison with Gradient Descent, have been extensively studied in the literature, e.g., \citet{nesterov2012efficiency,gower2015randomized,richtarik2016parallel}. While a coordinate variant generally has a less competitive worst-case convergence rate than that of the full update, it can benefit from a lower iteration cost, for which we refer to the related works and the references therein for an exposition. The purpose of showing the coordinate variants in this section is to demonstrate that the optimization principle of running an energy-conserving Hamiltonian Flow with zero initial velocity and restarts is indeed versatile and powerful.

\section{Parallel Frictionless Coordinate Hamiltonian Descent}~\label{sec:PCo-Hami}

In this section, we consider Parallel Frictionless Coordinate Hamiltonian Descent (Parallel Frictionless-CHD), 
Algorithm~\ref{alg:par}, 
which updates each coordinate simultaneously.

\begin{lemma} \label{lem:parallel-chd}
Denote
$\q{k}{i}  := \frac{b_i - \sum_{j \neq i} A_{i,j} x_{k}[j] }{A_{i,i}}$. 
For solving the quadratic problems \eqref{quadratic} via Parallel
Frictionless-CHD (Algorithm~\ref{alg:par}),
at each inner iteration $i$ of an outer iteration $k$, the update 
$(x_{k}^{(i)}, v_{k}^{(i)} ) = \hf_{\eta_{k,i}}^{(i)}(x_{k}, 0)$ is implemented as
\begin{align}
x_{k}^{(i)}[i] & = \q{k}{i} + \cos \left( \eta_{k,i} \sqrt{ A_{i,i} } \right) \left(x_k[i] - \q{k}{i} \right) \label{hihipchd} \\
v_{k}^{(i)}[i] & = - \sqrt{ A_{i,i} } \sin \left( \eta_{k,i} \sqrt{ A_{i,i} } \right) \left(x_k[i] - \q{k}{i} \right), \notag 
\end{align}
and for $j \neq i$, $x_{k}^{(i)}[j]  = x_{k}[j]$ and $v_{k}^{(i)}[j]  =0$.  
\end{lemma}

\begin{proof}

Applying Lemma~\ref{lem:1d} with $\alpha \gets A_{i,i}$, $\beta \gets b[i] - \sum_{j\neq i}^d A_{i,j} x[j]$,  $x \gets x_{k}$, $ t \gets \eta_{k,i}$ leads to the result.

\end{proof}

We now compare the update \eqref{hihipchd} with the Jacobi method and the weighted Jacobi method,
which are also classical methods for solving systems of equations \citep{hackbusch1994iterative,quarteroni2006numerical,iserles2009first}.
For the Jacobi method, at each iteration $k$, the update is:
\begin{align*} 
& \textbf{The Jacobi method:}  \quad
 x_{k+1}[i]  = \q{k}{i} =  \frac{b_i - \sum_{j \neq i} A_{i,j} x_{k}[j] }{A_{i,i}}, \forall i \in [d], 
\intertext{while the weighted Jacobi method has the following update:} 
& \textbf{The weighted Jacobi method:} \quad 
 x_{k+1}[i]  = c \, \q{k}{i} + (1-c) x_k[i], \forall i \in [d]. 
\end{align*}
It is evident that the Jacobi method and the weighted Jacobi method are both instances of 
Parallel Frictionless-CHD, with different values of the integration time $\eta_{k,i}$.
This result can serve as an interesting counterpart of Gauss-Seidel and the Successive Over-relaxation as Frictionless-CHD in the previous section. 
\begin{theorem}
The Jacobi method is Parallel Frictionless-CHD (Algorithm~\ref{alg:par}) with the integration time $\eta_{k,i}=\frac{1}{\sqrt{A_{i,i}}} \sin^{-1}(1) $, 
while the weighted Jacobi method is Parallel Frictionless-CHD with the integration time
$\eta_{k,i} = \frac{1}{\sqrt{A_{i,i}}} \cos^{-1}(1-c) $, where $c \in (0,2)$ is a parameter.
\end{theorem}

\begin{algorithm}[t]
\begin{algorithmic}[1]
\small
\caption{\textsc{Parallel Frictionless Coordinate Hamiltonian Descent
}
} \label{alg:par}{}
\STATE Input: an initial point $x_{1} \in \reals^{d}$, number of iterations $K$, and a scheme of integration time $\{\eta_{k,i} \}$.
\FOR{$k=1$ to $K$}
\STATE \text{//} Each update $i \in [d]$ in the inner loop can be executed in parallel. 
\FOR{$i=1$ to $d$}  
\STATE \quad Compute $(x_{k}^{(i)}, v_{k}^{(i)} ) = \hf_{\eta_{k,i}}^{(i)}(x_{k}, 0)$. 
\STATE \quad Set $x_{k+1}[i] = x_k^{(i)}[i]$.
\ENDFOR
\ENDFOR
\STATE Return $x_{K+1}$.
\end{algorithmic}
\end{algorithm}

However, unlike CHD and HD, Parallel CHD may not decrease the objective value at each outer iteration $k$ even if $v_{k}^{(i)}$ is non-zero, and it may not converge to the optimal point $x_*$. Nevertheless, 
Theorem~\ref{thm:parallel-chd} identifies a condition under which Parallel Frictionless-CHD converges.
\begin{theorem} \label{thm:parallel-chd}
Let $\eta_{k,i}=\eta_i  \neq \frac{1}{\sqrt{A_{i,i}}} \sin^{-1}(0) $. 
If the underlying matrix $A$ of the quadratic problem \eqref{quadratic} satisfies 
$ \left| A_{i,i} \left( 1 + \frac{2 \cos(\eta_i \sqrt{A_{i,i}})}{1-\cos(\eta_i \sqrt{A_{i,i}})} \right) \right| > \sum_{j \neq i} \left| A_{i,j} \right|, \forall i \in [d],$
then Parallel Frictionless-CHD (Algorithm~\ref{alg:par}) converges 
at an asymptotic rate $\rho\left( I_d - (I_d - M) D^{-1} A\right) < 1$,
where 
$\rho(\cdot)$ denotes the spectral radius of the underlying matrix, and
$M$
is a diagonal matrix
with $cos \left( \eta_{i} \sqrt{ A_{i,i} } \right)$ on its $i_{\mathrm{th}}$ element.
\end{theorem}

The proof of Theorem~\ref{thm:parallel-chd} requires the following result.
\begin{theorem} \label{thm:converge2} (e.g., Theorem 10.2 in \citet{iserles2009first})
Consider an iterative system,
$\delta_{k+1} = ( I - B A ) \delta_k$.
Suppose $A, B \succ 0$.
If $B^{-1} + (B^\top)^{-1} - A \succ 0$, then $\delta_k$ converges to zero
at an asymptotic rate $\rho (I-B A)<1$.
\end{theorem}

\begin{proof}[of Theorem~\ref{thm:parallel-chd}]
Using Lemma~\ref{lem:parallel-chd}, we know that
\begin{align*}
x_{k}^{(i)}[i] =  \q{k}{i} + \cos \left( \eta_{k,i} \sqrt{ A_{ii} } \right) \left(x_k[i] - \q{k}{i} \right).
\end{align*}

Denote $L$ as the strictly lower triangular part, $U$ as the strictly upper triangular part,
and $D$ as the diagonal part of the matrix $A$.
We can write the update of $x_{k+1}$ in terms of $L,U,D$ as: 
\begin{align} \label{54}
x_{k+1} = D^{-1} \left( b - L x_{k} - U x_k   \right) + M_k \left( x_k - D^{-1} \left( b - L x_{k} - U x_k   \right)  \right),
\end{align}
where 
$M_k := \text{Diag}\left( \cos \left( \eta_{k,1} \sqrt{ A_{1,1} } \right) ,\dots, \cos \left(  \eta_{k,d} \sqrt{ A_{d,d} } \right)\right)$. Using $b = A x_*$, we can further simplify \eqref{54} to
\begin{align*}
x_{k+1}  = & D^{-1} \left( A x_* - (A-D) x_k   \right) + M_k \left( x_k - D^{-1} \left( A x_* - (A-D) x_k   \right)  \right) \notag
\\  = & D^{-1} A (x_* -x_k ) + x_k - M_k D^{-1} A (x_* -x_k ) \notag \\
\iff
x_{k+1} - x_* & =   (I_d - D^{-1} A + M_k D^{-1} A ) (x_k - x_*). 
\end{align*}
Now, let $\eta_{k,i} = \eta_i \neq \frac{1}{\sqrt{A_{i,i}}} \sin^{-1}(0)$ be a constant.
Then, we have
$M_k = M:= \mathrm{Diag}\left( \cos \left( \eta_{1} \sqrt{ A_{1,1} } \right) ,\dots, \cos \left( \eta_{d} \sqrt{ A_{d,d} } \right)\right).$

To invoke Theorem~\ref{thm:converge2} with $B \leftarrow (I_d-M) D^{-1} \succ 0$, it suffices to identify conditions such that 
\begin{equation} \label{q}
\textstyle
\left(  (I_d-M) D^{-1} \right)^{-1} + \left( \left(  (I_d-M) D^{-1} \right)^\top \right)^{-1} - A \succ 0. 
\end{equation}
Let 
\begin{align*}
N &:= (I_d - M)^{-1} = \mathrm{Diag}\left( \frac{1}{1-\cos(\eta_1 \sqrt{A_{1,1}})}, \dots, \frac{1}{1-\cos(\eta_d \sqrt{A_{d,d}})} \right) \notag
\\ & = 
\mathrm{Diag}\left( 1 + \frac{\cos(\eta_1 \sqrt{A_{1,1}})}{1-\cos(\eta_1 \sqrt{A_{1,1}})}, \dots, 1 + \frac{\cos(\eta_d \sqrt{A_{d,d}})}{1-\cos(\eta_d \sqrt{A_{d,d}})} \right). \notag
\end{align*}
Then, the inequality \eqref{q} is equivalent to 
$2 D N - A \succ 0.$
If the matrix $2 D N - A$ is strictly diagonally dominant and its diagonal elements are positive, then $2 D N - A \succ 0$.
The condition of being strictly diagonally dominant 
is $ \left| A_{i,i} \left( 1 + \frac{ 2 \cos(\eta_i \sqrt{A_{i,i}})}{1-\cos(\eta_i \sqrt{A_{i,i}})} \right) \right| > \sum_{j \neq i} \left| A_{i,j} \right|, \forall i \in [d].$
Also, the condition that $\cos \left( \eta_{i} \sqrt{ A_{i,i} } \right) > -1, \forall i \in [d]$
guarantees that the diagonal elements of $2 D N - A$ are positive. 
We can now invoke Theorem~\ref{thm:converge2} to complete the proof.

\end{proof}

In the literature, it is known that a sufficient condition for the Jacobi method to converge is when $A$ is strictly diagonally-dominant, i.e.,
  $\left| A_{i,i} \right| > \sum_{j \neq i} \left| A_{i,j} \right|, \forall i \in [d]$, see e.g., Section 4.4.2 in \citet{hackbusch1994iterative}.  
The sufficient condition for convergence in Theorem~\ref{thm:parallel-chd} for Parallel CHD 
suggests that by choosing each $\eta_i \neq \frac{1}{\sqrt{A_{i,i}}}  \sin^{-1}(0)$ such that $\cos(\eta_i \sqrt{A_{i,i}})>0$, 
it leads to a weaker (sufficient) condition
than that of the Jacobi method.  
Hence, Parallel Frictionless-CHD not only unifies the Jacobi method and the weighted Jacobi method (both are instances of Parallel Frictionless-CHD) but also introduces new update schemes that may be more applicable than the existing ones.

\section{Conclusion} \label{outlooks}

In this paper, we introduce \textit{Frictionless Hamiltonian Descent} and analyze it for strongly convex quadratic functions. We show that Frictionless-HD attains a nontrivial acceleration rate analogous to that of Heavy Ball flow with resets in the context of strongly convex quadratic problems.  
Additionally, we propose \textit{Frictionless Coordinate Hamiltonian Descent} and its parallelizable variant and establish their connections with classical iterative methods for solving linear systems, including the Gauss-Seidel method, Successive Over-Relaxation, and the Jacobi method.  
Looking forward, we note that a key challenge in implementing the Hamiltonian Flow for \emph{general} functions arises from the fact that Hamiltonian equations do not typically admit simple analytical or closed-form solutions. This might necessitate the use of numerical integrators to approximate the Hamiltonian Flow, particularly in non-quadratic cases \citep{H06}. An elegant approach proposed in \citet{boyd2024optimization}, leveraging computer-assisted discretization, could prove valuable for simulating Frictionless-HD.  
Moreover, a wealth of well-established techniques and results exists in the literature on Hamiltonian Monte Carlo (HMC) methods for sampling \citep{HG14,chen2020fast,SG21,lee2021lower,kook2022sampling,monmarche2022hmc,chen2023does,bou2023mixing,noble2024unbiased,bou2024gist}. Given that Frictionless-HD can be viewed as a direct counterpart of HMC in optimization, investigating whether these results extend to Frictionless-HD may facilitate the development of novel optimization techniques.

\section*{Acknowledgements}
The author appreciates the support from NSF CCF-2403392 and also thanks Andre Wibisono for helpful discussions.

\bibliographystyle{plainnat}
\bibliography{main}

@article{okamura2026heavy,
  title={Heavy-ball differential equation achieves ${O}(\epsilon^{-7/4})$ convergence for nonconvex functions},
  author={Okamura, Kaito and Marumo, Naoki and Takeda, Akiko},
  journal={Optimization Letters},
  pages={1--18},
  year={2026},
  publisher={Springer}
}

@article{aujol2023convergence,
  title={Convergence rates of the {Heavy-Ball} method under the {{\L}}ojasiewicz property},
  author={Aujol, J-F and Dossal, Ch and Rondepierre, Aude},
  journal={Mathematical Programming},
  volume={198},
  number={1},
  pages={195--254},
  year={2023},
  publisher={Springer}
}

@article{sanz2021connections,
  title={The connections between {Lyapunov} functions for some optimization algorithms and differential equations},
  author={Sanz Serna, Jes{\'u}s Mar{\'\i}a and Zygalakis, Konstantinos C},
  journal={SIAM Journal on Numerical Analysis},
  volume={59},
  number={3},
  pages={1542--1565},
  year={2021},
  publisher={SIAM}
}

@article{zhu2026frictionless,
  title={Frictionless {Hamiltonian} descent with discretization and parallel optimization},
  author={Zhu, Heng and Wang, Jun-Kun},
  journal={Transactions on Machine Learning Research},
  year={2026}
}

@article{fu2025hamiltonian,
  title={Hamiltonian descent algorithms for optimization: Accelerated rates via randomized integration time},
  author={Fu, Qiang and Wibisono, Andre},
  journal={Advances in Neural Information Processing Systems (NeurIPS)},
  volume={38},
  year={2025}
}

@article{alamo2022restart,
  title={Restart of accelerated first-order methods with linear convergence under a quadratic functional growth condition},
  author={Alamo, Teodoro and Krupa, Pablo and Limon, Daniel},
  journal={IEEE Transactions on Automatic Control},
  volume={68},
  number={1},
  pages={612--619},
  year={2022},
  publisher={IEEE}
}

@article{bou2024gist,
  title={GIST: {Gibbs} self-tuning for locally adaptive {Hamiltonian Monte Carlo}},
  author={Bou-Rabee, Nawaf and Carpenter, Bob and Marsden, Milo},
  journal={arXiv preprint arXiv:2404.15253},
  year={2024}
}

@article{bou2023mixing,
  title={Mixing time guarantees for unadjusted {Hamiltonian Monte Carlo}},
  author={Bou-Rabee, Nawaf and Eberle, Andreas},
  journal={Bernoulli},
  volume={29},
  number={1},
  pages={75--104},
  year={2023},
  publisher={Bernoulli Society for Mathematical Statistics and Probability}
}

@article{nesterov2012efficiency,
  title={Efficiency of coordinate descent methods on huge-scale optimization problems},
  author={Nesterov, Yu},
  journal={SIAM Journal on Optimization},
  volume={22},
  number={2},
  pages={341--362},
  year={2012},
  publisher={SIAM}
}

@article{richtarik2016parallel,
  title={Parallel coordinate descent methods for big data optimization},
  author={Richt{\'a}rik, Peter and Tak{\'a}{\v{c}}, Martin},
  journal={Mathematical Programming},
  volume={156},
  pages={433--484},
  year={2016},
  publisher={Springer}
}

@article{gower2015randomized,
  title={Randomized iterative methods for linear systems},
  author={Gower, Robert M and Richt{\'a}rik, Peter},
  journal={SIAM Journal on Matrix Analysis and Applications},
  volume={36},
  number={4},
  pages={1660--1690},
  year={2015},
  publisher={SIAM}
}

@inproceedings{de2022born,
  title={Born-Infeld ({BI}) for {AI}: energy-conserving descent ({ECD}) for optimization},
  author={De Luca, Giuseppe Bruno and Silverstein, Eva},
  booktitle={{International Conference on Machine Learning (ICML)}},
  pages={4918--4936},
  year={2022},
  organization={PMLR}
}

@article{noble2024unbiased,
  title={Unbiased constrained sampling with Self-Concordant Barrier {Hamiltonian Monte Carlo}},
  author={Noble, Maxence and De Bortoli, Valentin and Durmus, Alain},
  journal={{Advances in Neural Information Processing Systems (NeurIPS)}},
  volume={37},
  year={2023}
}

@article{varga1962iterative,
  title={Matrix Iterative analysis},
  author={Varga, Richard S},
  journal={Springer},
  year={2000},
  publisher={Springer}
}

@article{beck2013convergence,
  title={On the convergence of block coordinate descent type methods},
  author={Beck, Amir and Tetruashvili, Luba},
  journal={SIAM journal on Optimization},
  volume={23},
  number={4},
  pages={2037--2060},
  year={2013},
  publisher={SIAM}
}

@article{golub1961chebyshev,
  title={Chebyshev semi-iterative methods, successive overrelaxation iterative methods, and second order Richardson iterative methods},
  author={Golub, Gene H and Varga, Richard S},
  journal={Numerische Mathematik},
  volume={3},
  number={1},
  pages={157--168},
  year={1961}
}

@inproceedings{Nemirovski84,
  title={Information-based complexity of convex programming},
  author={Arkadi Nemirovski},
  booktitle={Lecture notes},
  year={1994},
}

@book{greiner2003classical,
  title={Classical mechanics: systems of particles and Hamiltonian dynamics},
  author={Greiner, Walter},
  year={2003},
  publisher={Springer}
}

@article{muehlebach2021optimization,
  title={Optimization with momentum: Dynamical, control-theoretic, and symplectic perspectives},
  author={Muehlebach, Michael and Jordan, Michael I},
  journal={The Journal of Machine Learning Research},
  volume={22},
  number={1},
  pages={3407--3456},
  year={2021}
}

@article{betancourt2018symplectic,
  title={On symplectic optimization},
  author={Betancourt, Michael and Jordan, Michael I and Wilson, Ashia C},
  journal={arXiv preprint arXiv:1802.03653},
  year={2018}
}

@inproceedings{AGZ21,
  title = {Acceleration via Fractal Learning Rate Schedules},
  author = {Naman Agarwal and Surbhi Goel and Cyril Zhang},
  booktitle = {{International Conference on Machine Learning (ICML)}},
  year = {2021}
}

@article{young1954iterative,
  title={Iterative methods for solving partial difference equations of elliptic type},
  author={Young, David},
  journal={Transactions of the American Mathematical Society},
  volume={76},
  number={1},
  pages={92--111},
  year={1954}
}

@inproceedings{boyd2024optimization,
  title={Optimization Algorithm Design via Electric Circuits},
  author={Boyd, Stephen P and Parshakova, Tetiana and Ryu, Ernest K and Suh, Jaewook J},
  booktitle={{Advances in Neural Information Processing Systems (NeurIPS)}},
  year={2024}
}

@article{ostrowski1954linear,
  title={On the linear iteration procedures for symmetric matrices},
  author={Ostrowski, Alexander M},
  journal={Rend. Mat. Appl.},
  volume={14},
  pages={140--163},
  year={1954}
}

@book{quarteroni2006numerical,
  title={Numerical mathematics},
  author={Quarteroni, Alfio and Sacco, Riccardo and Saleri, Fausto},
  volume={37},
  year={2006},
  publisher={Springer Science \& Business Media}
}

@article{vishnoi2013lx,
  title={Lx= b},
  author={Vishnoi, Nisheeth K and others},
  journal={Foundations and Trends{\textregistered} in Theoretical Computer Science},
  volume={8},
  number={1--2},
  pages={1--141},
  year={2013},
  publisher={Now Publishers, Inc.}
}

@book{iserles2009first,
  title={A first course in the numerical analysis of differential equations},
  author={Iserles, Arieh},
  number={44},
  year={2009},
  publisher={Cambridge university press}
}

@book{hackbusch1994iterative,
  title={Iterative solution of large sparse systems of equations},
  author={Hackbusch, Wolfgang},
  volume={95},
  year={1994},
  publisher={Springer}
}

@article{leventhal2010randomized,
  title={Randomized methods for linear constraints: convergence rates and conditioning},
  author={Leventhal, Dennis and Lewis, Adrian S},
  journal={Mathematics of Operations Research},
  volume={35},
  number={3},
  pages={641--654},
  year={2010},
  publisher={INFORMS}
}

@article{maddison2018hamiltonian,
  title={Hamiltonian descent methods},
  author={Maddison, Chris J and Paulin, Daniel and Teh, Yee Whye and O'Donoghue, Brendan and Doucet, Arnaud},
  journal={arXiv preprint arXiv:1809.05042},
  year={2018}
}

@article{wright2015coordinate,
  title={Coordinate descent algorithms},
  author={Wright, Stephen J},
  journal={Mathematical programming},
  volume={151},
  number={1},
  pages={3--34},
  year={2015},
  publisher={Springer}
}

@article{diakonikolas2021generalized,
  title={Generalized momentum-based methods: A {Hamiltonian}  perspective},
  author={Diakonikolas, Jelena and Jordan, Michael I},
  journal={SIAM Journal on Optimization},
  volume={31},
  number={1},
  pages={915--944},
  year={2021},
  publisher={SIAM}
}

@article{francca2020conformal,
  title={Conformal symplectic and relativistic optimization},
  author={Fran{\c{c}}a, Guilherme and Sulam, Jeremias and Robinson, Daniel and Vidal, Ren{\'e}},
  journal={{Advances in Neural Information Processing Systems (NeurIPS)}},
  volume={33},
  pages={16916--16926},
  year={2020}
}

@article{salvatier2016probabilistic,
  title={Probabilistic programming in Python using PyMC3},
  author={Salvatier, John and Wiecki, Thomas V and Fonnesbeck, Christopher},
  journal={PeerJ Computer Science},
  volume={2},
  pages={e55},
  year={2016},
  publisher={PeerJ Inc.}
}

@article{carpenter2017stan,
  title={Stan: A probabilistic programming language},
  author={Carpenter, Bob and Gelman, Andrew and Hoffman, Matthew D and Lee, Daniel and Goodrich, Ben and Betancourt, Michael and Brubaker, Marcus A and Guo, Jiqiang and Li, Peter and Riddell, Allen},
  journal={Journal of statistical software},
  volume={76},
  year={2017},
  publisher={NIH Public Access}
}

@article{wang2022accelerating,
  title={Accelerating {Hamiltonian Monte Carlo} via {Chebyshev} Integration Time},
  author={Wang, Jun-Kun and Wibisono, Andre},
  journal={International Conference on Learning Representations (ICLR)},
  year={2023}
}

@book{H06,
  title={Geometric Numerical Integration
Structure-Preserving Algorithms for Ordinary Differential Equations},
  author={Hairer, Ernst and Wanner, Gerhard and Lubich, Christian},
  year={2006},
  publisher={Springer Series in Computational Mathematics}
}

@article{chen2023does,
  title={When does {Metropolized} {Hamiltonian Monte Carlo} provably outperform {Metropolis}-adjusted {Langevin} algorithm?},
  author={Chen, Yuansi and Gatmiry, Khashayar},
  journal={arXiv preprint arXiv:2304.04724},
  year={2023}
}

@article{o2019hamiltonian,
  title={Hamiltonian descent for composite objectives},
  author={O'Donoghue, Brendan and Maddison, Chris J},
  journal={{Advances in Neural Information Processing Systems (NeurIPS)}},
  volume={32},
  year={2019}
}

@inproceedings{teel2019first,
  title={First-order optimization algorithms with resets and {Hamiltonian} flows},
  author={Teel, Andrew R and Poveda, Jorge I and Le, Justin},
  booktitle={2019 IEEE 58th Conference on Decision and Control (CDC)},
  pages={5838--5843},
  year={2019},
  organization={IEEE}
}

@article{chen2020fast,
  title={Fast mixing of {Metropolized} {Hamiltonian Monte Carlo}: Benefits of multi-step gradients},
  author={Chen, Yuansi and Dwivedi, Raaz and Wainwright, Martin J and Yu, Bin},
  journal={The Journal of Machine Learning Research},
  volume={21},
  number={1},
  pages={3647--3717},
  year={2020},
  publisher={JMLRORG}
}

@article{kook2022sampling,
  title={Sampling with {Riemannian Hamiltonian Monte Carlo} in a constrained space},
  author={Kook, Yunbum and Lee, Yin-Tat and Shen, Ruoqi and Vempala, Santosh},
  journal={{Advances in Neural Information Processing Systems (NeurIPS)}},
  volume={35},
  pages={31684--31696},
  year={2022}
}

@article{lee2021lower,
  title={Lower bounds on {Metropolized} sampling methods for well-conditioned distributions},
  author={Lee, Yin Tat and Shen, Ruoqi and Tian, Kevin},
  journal={{Advances in Neural Information Processing Systems (NeurIPS)}},
  volume={34},
  pages={18812--18824},
  year={2021}
}

@ARTICLE{WW23a,
  title={Continuized Acceleration for Quasar Convex Functions in Non-Convex Optimization},
  author={Jun-Kun Wang and Andre Wibisono},
  journal   = {International Conference on Learning Representations (ICLR)},
  year      = {2023}
}

@inproceedings{goujaud2022super,
  title={Super-acceleration with cyclical step-sizes},
  author={Goujaud, Baptiste and Scieur, Damien and Dieuleveut, Aymeric and Taylor, Adrien B and Pedregosa, Fabian},
  booktitle={International Conference on Artificial Intelligence and Statistics (AISTATS)},
  pages={3028--3065},
  year={2022},
  organization={PMLR}
}

@inproceedings{kelner2022big,
  title={Big-Step-Little-Step: Efficient Gradient Methods for Objectives with Multiple Scales},
  author={Kelner, Jonathan and Marsden, Annie and Sharan, Vatsal and Sidford, Aaron and Valiant, Gregory and Yuan, Honglin},
  booktitle={Conference on Learning Theory (COLT)},
  pages={2431--2540},
  year={2022},
  organization={PMLR}
}

@inproceedings{paquette2021sgd,
  title={{SGD} in the large: Average-case analysis, asymptotics, and stepsize criticality},
  author={Paquette, Courtney and Lee, Kiwon and Pedregosa, Fabian and Paquette, Elliot},
  booktitle={Conference on Learning Theory (COLT)},
  pages={3548--3626},
  year={2021},
  organization={PMLR}
}

@article{dixit2023accelerated,
  title={Accelerated gradient methods for nonconvex optimization: Escape trajectories from strict saddle points and convergence to local minima},
  author={Dixit, Rishabh and Gurbuzbalaban, Mert and Bajwa, Waheed U},
  journal={arXiv preprint arXiv:2307.07030},
  year={2023}
}

@article{akande2024momentum,
  title={Momentum-based minimization of the {Ginzburg-Landau} functional on {Euclidean} spaces and graphs},
  author={Akande, Oluwatosin and Dondl, Patrick and Gupta, Kanan and Onwunta, Akwum and Wojtowytsch, Stephan},
  journal={arXiv preprint arXiv:2501.00389},
  year={2024}
}

@article{jiang2022dissipation,
  title={{On the Dissipation of Ideal {Hamiltonian Monte Carlo} Sampler}},
  author={Jiang, Qijia},
  journal={{STAT}},
  volumn={12},
  year={2023},
  Publisher={Wiley}
}

@ARTICLE{DST21,
  title = {Acceleration Methods},
  author = {Alexandre d'Aspremont and Damien Scieur and Adrien Taylor},
  journal = {Foundations and Trends in Optimization },
  year = {2021}
}

@ARTICLE{Duane87,
  title = {{Hybrid Monte Carlo}},
  author = {Duane, Simon and Kennedy, A. D. and Pendleton, Brian J. and Roweth, Duncan},
  journal = {Physics Letters B},
  year = {1987}
}

@incollection{neal2011mcmc,
  author    = {Neal, Radford M.},
  title     = {{MCMC} Using {H}amiltonian Dynamics},
  booktitle = {Handbook of {M}arkov Chain {M}onte Carlo},
  editor    = {Brooks, Steve and Gelman, Andrew and Jones, Galin L. and Meng, Xiao-Li},
  publisher = {Chapman \& Hall/CRC},
  year      = {2011},
  pages     = {113--162}
}

@article{HG14,
  title={The {No-U-Turn} Sampler: Adaptively Setting Path Lengths in {Hamiltonian Monte Carlo}},
  author={Hoffman, Matthew D and Gelman, Andrew},
  journal={Journal of Machine Learning Research},
  volume={15},
  pages={1351--1381},
  year={2014}
}

@ARTICLE{AGV21,
  title={Convergence rates for the {Heavy-Ball} continuous dynamics for non-convex optimization, under {Polyak-\L{}ojasiewicz} conditioning},
  author={Vassilis Apidopoulos and Nicolo Ginatta and Silvia Villa},
  journal   = {arXiv:2107.10123},
  year      = {2021}
}

@InProceedings{pmlr-v336-wang26c,
  title =    {Accelerated Convex Optimization via Hamiltonian Dynamics with Deterministic Integration Time},
  author =       {Wang, Xiuyuan and Srinivasan, Vishwak and Fu, Qiang and Mitra, Siddharth and Wilson, Ashia and Wibisono, Andre},
  booktitle =    {Conference on Learning Theory (COLT)},
  pages =    {6692--6742},
  year =   {2026}
}

@inproceedings{WLA21,
  title={A Modular Analysis of Provable Acceleration via {Polyak's} Momentum: Training a Wide {ReLU} Network and a Deep Linear Network},
  author={Jun-Kun Wang and Chi-Heng Lin and Jacob Abernethy},
  booktitle   = {International Conference on Machine Learning (ICML)},
  year      = {2021}
}

@inproceedings{SP20,
  title={Universal Average-Case Optimality of {Polyak} Momentum},
  author={Damien Scieur and Fabian Pedregosa},
  booktitle   = {International Conference on Machine Learning (ICML)},
  year      = {2020}
}

@article{SG21,
  title={Hamiltonian dynamics with non-newtonian momentum for rapid sampling},
  author={Ver Steeg, Greg and Galstyan, Aram},
  journal={{Advances in Neural Information Processing Systems (NeurIPS)}},
  volume={34},
  pages={11012--11025},
  year={2021}
}

@article{wibisono2016variational,
  title={A variational perspective on accelerated methods in optimization},
  author={Wibisono, Andre and Wilson, Ashia C and Jordan, Michael I},
  journal={Proceedings of the National Academy of Sciences},
  volume={113},
  number={47},
  pages={E7351--E7358},
  year={2016},
  publisher={National Acad Sciences}
}

@inproceedings{NB15,
  title={From Averaging to Acceleration, There is Only a
Step-size},
  author={Nicolas Flammarion and Francis Bach},
  booktitle   = {{Conference on Learning Theory (COLT)}},
  year      = {2015}
}

@article{kassing2024polyak,
  title={Polyak's {Heavy Ball} Method Achieves Accelerated Local Rate of Convergence under {Polyak-Lojasiewicz} Inequality},
  author={Kassing, Sebastian and Weissmann, Simon},
  journal={arXiv preprint arXiv:2410.16849},
  year={2024}
}

@inproceedings{HL17,
  title={Dissipativity Theory for {Nesterov’s} Accelerated Method},
  author={Bin Hu and Laurent Lessard},
  booktitle = {International Conference on Machine Learning (ICML)},
  year      = {2017}
}

@inproceedings{nesterov1983method,
  title={A method for unconstrained convex minimization problem with the rate of convergence { $O(1/k^{2})$}},
  author={Nesterov, Yurii},
  booktitle={Doklady AN USSR},
  volume={269},
  pages={543--547},
  year={1983}
}

@article{shi2022understanding,
  title={Understanding the acceleration phenomenon via high-resolution differential equations},
  author={Shi, Bin and Du, Simon S and Jordan, Michael I and Su, Weijie J},
  journal={Mathematical Programming},
  volume={195},
  number={1},
  pages={79--148},
  year={2022},
  publisher={Springer}
}

@article{wilson2016lyapunov,
  title={A {Lyapunov} Analysis of Accelerated Methods in Optimization},
  author={Wilson, Ashia C and Recht, Ben and Jordan, Michael I},
  journal={Journal of Machine Learning Research},
  volume={22},
  number={113},
  pages={1--34},
  year={2021}
}

@article{wang2024no,
  title={No-regret dynamics in the {Fenchel} game: A unified framework for algorithmic convex optimization},
  author={Wang, Jun-Kun and Abernethy, Jacob and Levy, Kfir Y},
  journal={Mathematical Programming},
  volume={205},
  number={1},
  pages={203--268},
  year={2024},
  publisher={Springer}
}

@article{krichene2015accelerated,
  title={Accelerated mirror descent in continuous and discrete time},
  author={Krichene, Walid and Bayen, Alexandre and Bartlett, Peter L},
  journal={{Advances in Neural Information Processing Systems (NeurIPS)}},
  volume={28},
  year={2015}
}

@book{nesterov2018lectures,
  title={Lectures on convex optimization},
  author={Nesterov, Yurii},
  volume={137},
  year={2018},
  publisher={Springer}
}

@ARTICLE{SRBA17,
  title={Integration methods and optimization algorithms},
  author={Damien Scieur and Vincent Roulet and Francis Bach and Alexandre d'Aspremont},
   pages        = {1109--1118},
   volume = {30},
  journal   = {Advances in Neural Information Processing Systems (NeurIPS)},
  year      = {2017}
}

@ARTICLE{DO19,
  title={The Approximate Duality Gap Technique: A Unified Theory of First-Order Methods},
  author={Jelena Diakonikolas and Lorenzo Orecchia},
  journal   = {SIAM Journal on Optimization},
    volume       = {29},
  number       = {1},
  pages        = {660--689},
  year      = {2019}
}

@article{zhang2018direct,
  title={Direct {Runge-Kutta} discretization achieves acceleration},
  author={Zhang, Jingzhao and Mokhtari, Aryan and Sra, Suvrit and Jadbabaie, Ali},
  journal={{Advances in Neural Information Processing Systems (NeurIPS)}},
  volume={31},
  year={2018}
}

@ARTICLE{P64,
  title={Some methods of speeding up the convergence of iteration methods},
  author={B.T. Polyak},
  journal   = {USSR Computational Mathematics and Mathematical Physics},
  volume = {4},
number = {5},
pages = {1-17},
  year      = {1964}
}

@article{monmarche2022hmc,
  title={{HMC} and {Langevin} united in the unadjusted and convex case},
  author={Monmarch{\'e}, Pierre},
  journal={arXiv preprint arXiv:2202.00977},
  year={2022}
}

@article{wang2018acceleration,
  title={Acceleration through optimistic no-regret dynamics},
  author={Wang, Jun-Kun and Abernethy, Jacob D},
  journal={Advances in Neural Information Processing Systems (NeurIPS)},
  volume={31},
  year={2018}
}

@book{Arnold1989, 
    place={New York}, 
    title={Mathematical methods of classical mechanics}, 
    publisher={Springer}, 
    author={Arnold, V. I.}, 
    year={1989}
}

@inproceedings{SuBoydCandes2014_differential,
  title = {A Differential Equation for Modeling {{Nesterov}}'s Accelerated Gradient Method: {{Theory}} and Insights},
  author = {Su, Weijie and Boyd, Stephen and Cand{\`e}s, Emmanuel J.},
  year = {2014},
  booktitle = {{Advances in Neural Information Processing Systems (NeurIPS)}}
}

@article{ZhangOrvietoDaneshmand2021_rethinking,
  title = {Rethinking the Variational Interpretation of Accelerated Optimization Methods},
  author = {Zhang, Peiyuan and Orvieto, Antonio and Daneshmand, Hadi},
  year = {2021},
  journal = {{Advances in Neural Information Processing Systems (NeurIPS)}}
}

@article{AttouchChbaniPeypouquetRedont2018_fast,
  title = {Fast Convergence of Inertial Dynamics and Algorithms with Asymptotic Vanishing Viscosity},
  author = {Attouch, Hedy and Chbani, Zaki and Peypouquet, Juan and Redont, Patrick},
  year = {2018},
  journal = {Mathematical Programming},
  volume = {168},
  number = {1},
  pages = {123--175},
  urldate = {2022-01-05},
  langid = {english}
}

@article{FrancaRobinsonVidal2018_admm,
  title = {{{ADMM}} and Accelerated {{ADMM}} as Continuous Dynamical Systems},
  author = {Fran{\c c}a, Guilherme and Robinson, Daniel and Vidal, Ren{\'e}},
  year = {2018},
  journal = {{International Conference on Machine Learning (ICML)}},
  urldate = {2022-01-05}
}

@article{BotCsetnek2018_convergence,
  title = {Convergence Rates for Forward--Backward Dynamical Systems Associated with Strongly Monotone Inclusions},
  author = {Bo{\c t}, Radu Ioan and Csetnek, Ern{\"o} Robert},
  year = {2018},
  journal = {Journal of Mathematical Analysis and Applications},
  volume = {457},
  number = {2},
  pages = {1135--1152}
}

@article{FrancaJordanVidal2021_dissipative,
  title = {On Dissipative Symplectic Integration with Applications to Gradient-Based Optimization},
  author = {Fran{\c c}a, Guilherme and Jordan, Michael I and Vidal, Ren{\'e}},
  year = {2021},
  journal = {Journal of Statistical Mechanics: Theory and Experiment},
  volume = {2021},
  number = {4},
  pages = {043402},
  publisher = {IOP Publishing}
}

@article{SuhRohRyu2022_continuoustime,
  title = {Continuous-Time Analysis of {{AGM}} via Conservation Laws in Dilated Coordinate Systems},
  author = {Suh, Jaewook J. and Roh, Gyumin and Ryu, Ernest K.},
  year = {2022},
  journal = {International Conference on Machine Learning},
  copyright = {Creative Commons Attribution Non Commercial No Derivatives 4.0 International}
}

@article{even2021continuized,
  title={A continuized view on {N}esterov acceleration for stochastic gradient descent and randomized gossip},
  author={Even, Mathieu and Berthier, Rapha{\"e}l and Bach, Francis and Flammarion, Nicolas and Gaillard, Pierre and Hendrikx, Hadrien and Massouli{\'e}, Laurent and Taylor, Adrien},
  journal={arXiv preprint arXiv:2106.07644},
  year={2021}
}

@article{necoara2019linear,
  title={Linear convergence of first order methods for non-strongly convex optimization},
  author={Necoara, Ion and Nesterov, Yu and Glineur, Francois},
  journal={Mathematical programming},
  volume={175},
  number={1},
  pages={69--107},
  year={2019},
  publisher={Springer}
}

@article{attouch2022first,
  title={First-order optimization algorithms via inertial systems with {Hessian} driven damping},
  author={Attouch, Hedy and Chbani, Zaki and Fadili, Jalal and Riahi, Hassan},
  journal={Mathematical Programming},
  pages={113–155},
  year={2022},
  publisher={Springer}
}

@inproceedings{muehlebach2019dynamical,
  title={A dynamical systems perspective on {Nesterov} acceleration},
  author={Muehlebach, Michael and Jordan, Michael},
  booktitle={International Conference on Machine Learning (ICML)},
  pages={4656--4662},
  year={2019},
  organization={PMLR}
}



\end{document}